\documentclass{sinst}

\usepackage{amsbsy,amsfonts,amsmath,amssymb,amsthm}
\usepackage{array}
\usepackage{bm}
\usepackage{color}
\usepackage{enumerate}
\usepackage{mathrsfs}
\usepackage{latexsym}
\usepackage[colorlinks=true]{hyperref}
\usepackage{mathtools}
\mathtoolsset{showonlyrefs, showmanualtags}
\hypersetup{urlcolor=blue, citecolor=blue, linkcolor=blue}

\definecolor{wineRed}{rgb}{0.7,0,0.3}
\definecolor{grandBleu}{rgb}{0,0,0.8}
\definecolor{darkGreen}{rgb}{0,0.4,0}
\definecolor{blueViolet}{rgb}{0.4,0,1.0}
\definecolor{bloodOrange}{rgb}{0.85,0.05,0}
\definecolor{mycolor}{rgb}{0.8,0,0.2}

\usepackage[textsize=small]{todonotes}
\usepackage{cite}

\usepackage{comment}

\DeclareMathAlphabet{\mathpzc}{OT1}{pzc}{m}{it}

\usepackage[textsize=small]{todonotes}
\numberwithin{equation}{section}

\theoremstyle{plain}
\newtheorem{mainThm}{Main Theorem} 
\newtheorem{lemma}{Lemma}[section]

\newtheorem{theorem}{Theorem}

\theoremstyle{definition}
\newtheorem{definition}{Definition}
\newtheorem{rem}{Remark}
\newtheorem{notn}{Notation}
\newtheorem{ex}{Example}

\def\N{\mathbb{N}}
\def\R{\mathbb{R}}

\def\ts{\textstyle}

\begin{document}
\label{page:t}
\thispagestyle{plain}

\title{
A GRADIENT SYSTEM BASED ON ANISOTROPIC \\[-0.5ex] MONOCHROME IMAGE PROCESSING WITH \\[-0.5ex] ORIENTATION AUTO-ADJUSTMENT 
\vspace{-4ex}
}
\author{Harbir Antil
}
\affiliation{~~\\[-4ex] Department of Mathematical Sciences and \\[-0.5ex] the Center for Mathematics and Artificial Intelligence (CMAI), 
\\
George Mason University, Fairfax, VA 22030, USA\\[-0.5ex]}
\email{hantil@gmu.edu}

\sauthor{~~\\[-4ex]{\sc Daiki Mizuno}
}
\saffiliation{~~\\[-5ex] Division of Mathematics and Informatics, \\[-0.5ex] Graduate School of Science and Engineering, Chiba University, \\[-0.5ex] 1-33, Yayoi-cho, Inage-ku, 263-8522, Chiba, Japan\\[-0.5ex]}
\semail{d-mizuno@chiba-u.jp}
\tauthor{~~\\[-4ex]{\sc Ken Shirakawa}
}
\taffiliation{~~\\[-5ex] Department of Mathematics, Faculty of Education, Chiba University \\[-0.5ex] 1--33 Yayoi-cho, Inage-ku, 263--8522, Chiba, Japan\\[-0.5ex]}
\temail{sirakawa@faculty.chiba-u.jp}

\fauthor{~~\\[-3.5ex]{\sc Naotaka Ukai}
}
\faffiliation{~~\\[-5ex] Division of Mathematics and Informatics, \\[-0.5ex] Graduate School of Science and Engineering, Chiba University, \\[-0.5ex] 1-33, Yayoi-cho, Inage-ku, 263-8522, Chiba, Japan\\[-0.5ex]}
\femail{24wd0101@student.gs.chiba-u.jp}
\vspace{-1.75ex}

\footcomment{
}
\maketitle
\vspace{-2ex}

\noindent
{\bf Abstract.}
This paper is devoted to the mathematical analysis of a system of pseudo-parabolic partial differential equations governed by an energy functional, associated with anisotropic monochrome image processing. The energy functional is based on the one proposed by [Berkels et al. \emph{Cartoon extraction based on anisotropic image classification}, SFB 611, 2006], which incorporates an orientation auto-adjustment mechanism, and our energy is a simplified version which reduces the order of derivatives to improve computational efficiency. The aim of this paper is to establish a stable minimization process that addresses some instability of the algorithm caused by the reduction of derivative order. As a part of the study, we here prove Main Theorems concerned with the well-posedness of the pseudo-parabolic system, and provide a mathematically rigorous guarantee for the stability of the image denoising process derived from our system.
\pagebreak

\section*{Introduction}

Let $ 0 < T < \infty $ be a fixed constant, let $ \Omega \subset \R^2 $ be a fixed two-dimensional domain with a Lipschitz boundary $ \Gamma := \partial \Omega $, and let $ Q := (0, T) \times \Omega $ and $ \Sigma := (0, T) \times \Gamma $ be product domains of the time-space variable $(t, x)$. Let $ \kappa $, $ \mu $, {$ \nu $, and $\lambda$} are fixed positive constants, and let $ p > 2 $ is a fixed constant.

In this paper, we consider the following system of pseudo-parabolic PDEs, denoted by (S).

\begin{align}
    \mbox{(S):} ~~~~&
    \nonumber
    \\
    & 
    \begin{cases}\label{alpha00001}
        \partial_t \alpha -\mathit{\Delta} \bigl( \alpha +\kappa \partial_t \alpha \bigr) +\partial \gamma(R(\alpha)\nabla u) \cdot R(\alpha +{\ts \frac{\pi}{2}}) \nabla u \ni 0 \mbox{ in $ Q $,}
        \\[1ex]
        \alpha  = \partial_t \alpha = 0 \mbox{ on $ \Sigma $, ~~ $ \alpha(0, x) = \alpha_0(x) $, $ x\in \Omega  $,}
    \end{cases}
    \\
    & 
    \begin{cases}\label{u00001}
        \partial_t u -\mathrm{div} \bigl( \, {^\top R} (\alpha) \partial \gamma(R(\alpha) \nabla u) +\nu |\nabla u|^{p -2} \nabla u +\mu \nabla \partial_t u \, \bigr)
        \\
        \hspace{4ex} +\lambda(u -u_\mathrm{org}) \ni 0 \mbox{ in $ Q  $,}
        \\[1ex]
        u =  \partial_t u = 0 \mbox{ on $ \Sigma $, ~~ $ u(0, x) = u_0(x) $, $ x\in \Omega  $.}
    \end{cases}
\end{align}
This system {is a gradient flow of the following energy-functional:}
\begin{gather}
    E: [u, \alpha] \in W^{1, p}_0(\Omega) \times H_0^1(\Omega) \mapsto E(u, \alpha) := \frac{1}{2} \int_\Omega |\nabla \alpha|^2 \, dx +\frac{\nu}{p}\int_\Omega |\nabla u|^p \, dx 
    \nonumber
    \\
+\int_\Omega \gamma(R(\alpha) \nabla u) \, dx +\frac{\lambda}{2}\int_\Omega |u -u_\mathrm{org}|^2 \, dx \in [0, \infty).
    \label{energy01}
\end{gather}
{Hence, from practical viewpoint, it will correspond to a continuum limit of a discrete minimization scheme.
The energy $ E $ given in \eqref{energy01}} is a modified version of a monochrome image-denoising problem, developed by Berkels et al. \cite{berkels2006cartoon}, given as: 
{
\begin{gather}
    [u, \alpha] \in W_0^{1,p}(\Omega) \times \bigl(H^2(\Omega) \cap H^{1}_0(\Omega) \bigr) \mapsto E(u, \alpha) +\frac{1}{2} \int_\Omega |\mathit{\Delta} \alpha|^2 \, dx, 
    \nonumber
    \\
    \mbox{under $ p = 2 $.}
    \label{energy02}
\end{gather}
}
The original work \cite{berkels2006cartoon} focuses on special images, such as satellite photos of cityscapes, in which numerous unique polygonal (rectangular) patterns appear. In this context, $ \Omega $ is the spatial domain of the monochrome image region. {The unknown $ u $ is the order parameter which denotes the degree of gray-scale in $ \Omega $. Moreover, $ \alpha $ is an order parameter which helps to accurately capture the} focused polygonal patterns, together with a fixed function of anisotropy $ \gamma $ defined on $ \R^2 $. Indeed, the anisotropic function $\gamma$ is set as a non-Euclidean norm, so that the focused polygon matches the shape of its closed ball (Wulff shape), and the order parameter $\alpha$ is set as the data of angles, representing the orientation of the polygonal patterns distributed in $\Omega$.  

The noteworthy idea in  \cite{berkels2006cartoon} is to adjust the orientation of polygonal patterns, by using the unknown $ \alpha $ and the function of anisotropy $ \gamma $, in addition to the conventional order parameter $ u $ of gray-scale. From the perspective of accuracy improvement, the effectiveness of this idea is confirmed {by numerical results in} \cite{berkels2006cartoon}. 

In the meantime, from \eqref{energy01} and \eqref{energy02}, we can see that \cite{berkels2006cartoon} requires the potential of double-Laplacian $(-{\mit \Delta})^2 \alpha$ for the angle $ \alpha $. This extra potential {is meant to stabilize the minimization process, especially the complex interaction due to} $ \int_\Omega \gamma(R(\alpha)\nabla u) \, dx $ in \eqref{energy01}. {However, such} higher-order term would lead to a decrease in computational efficiency, with multiple boundary conditions. 

In view of this, we here focus on the reduced energy $ E = E(u, \alpha) $ given in \eqref{energy01}, and derive the system (S) as the following gradient system of pseudo-parabolic type:
\begin{gather*}
    -\left[ \begin{matrix}
        I +\kappa (-{\mit \Delta}_0) & O
        \\[0ex]
        O & I +\mu (-{\mit \Delta}_0)
    \end{matrix} \right] \frac{d}{dt} \left[ \begin{matrix}
        u(t) \\ \alpha(t)
    \end{matrix} \right] = \frac{\delta}{\delta [u, \alpha]} E(u(t), \alpha(t)) 
    \\[1ex]
    \mbox{in $ [L^2(\Omega)]^2 $, ~for $ t \in (0, T) $,}
\end{gather*}
with the functional derivative $ \frac{\delta}{\delta [u, \alpha]} E $, and the maximal monotone operator $ -{\mit \Delta}_0 $ of Laplacian under Dirichlet-zero boundary condition. {Such a gradient flow has been recently used by the authors in the context of material science problems recently in \cite{HAntil_DMizuno_KShirakawa_2024a}.}

From the derivation method, it could be expected that the solution to the system (S) realizes the energy-dissipation of $ E = E(u, \alpha) $. Also, the reduction of double Laplacian $ (-{\mit \Delta})^2 \alpha $ is expected to lead to the reduction of the order of spatial-derivative. Indeed, at the level of time-discretization, we will see that the scheme is formulated as a system of elliptic boundary value problems with (single) Dirichlet-zero boundary condition (see Section 3). Furthermore, since the system (S) corresponds to the continuous limit of discrete algorithms, its mathematical analysis {will provide} valuable information in the stabilization of numerical algorithm, with respect to parameter-choices such as time-step-size, spatial mesh-size, and so on.

Based on these, we set the goal to prove the following Main Theorems. 
\begin{description}
    \item[Main Theorem 1:]Existence and energy-dissipation for the system (S).
        \vspace{-1ex}
    \item[Main Theorem 2:]Uniqueness and continuous dependence when $ \gamma \in C^{1, 1}(\R^2) \cap C^2(\R^2) $.
\end{description}
        The first Main Theorem 1 will be proved by means of the time-discretization method. Since the time-discretization scheme has the uniqueness, the lack of uniqueness result in Main Theorem 1 would not be so crucial in practical applications. Meanwhile, the second Main Theorem 2 implies that we will guarantee more stable property of minimizing algorithm if we set the anisotropy $ \gamma $ sufficiently smooth. In any result, it should be noted that the nonlinear diffusion $ -\mu \, \mathrm{div} \bigl(|\nabla u|^{p -2} \nabla u \bigr) $ might bring some negative effect on computational efficiency. This issue {will be} one of the significant interests for future research, as well as the effect of the pseudo-parabolic nature of our system (S).

        {
\begin{rem}\label{exRem.01}
    Throughout this paper, the setting $ p > 2 $ is crucial in the derivation of:
    \begin{description}
        \item[$\sharp\,1)$] the energy-dissipation for the time-discretization scheme, involving the 2nd order elliptic systems, as used in the proof of Main Theorem 1;
        \item[$\sharp\,2)$]the Gronwall-type inequality used in the proof of Main Theorem 2. 
    \end{description}
    In this light, the setting $ p = 2 $ as in \eqref{energy02} would represent a critical condition obtaining the results stated in $\sharp\,1)$ and $\sharp\,2)$.
\end{rem}
}

\noindent
\textbf{Outline.} Section 1 provides the preliminaries, and based on this, the Main Theorems are stated in Section 2.  For the proofs of Main Theorems, Section 3 sets up a time-discretization scheme for (S), and prepares some auxiliary results. Finally, utilizing the auxiliary results, the Main Theorems 1 and 2 are proved in {Sections} 4 and 5, respectively.

\section{Preliminaries}
First, let us explain the notations used throughout.
\begin{notn}[\bf{Real analysis}]\label{notn1}We define:
  \[a \lor b : = \max\{ a , b \} \mbox{ and } a \land b : = \min\{ a , b \} , \mbox{ for all } a , b \in [-\infty, \infty],\]
  and especially, we note:
  \[[a]^+:= a \lor 0 \mbox{ and } [a]^-:= -(a \land 0), \mbox{ for all } a \in [-\infty,\infty].\]

Let $d \in \N$ be fixed dimension. We denote by $ | x | $ and $ x \cdot y $ the Euclidean norm of $ x \in \R^d$ and the standard scalar product of $ x , y \in \R^d$, respectively, i.e.,
\begin{gather*}
  | x | : = \sqrt{x_1^2 + \cdots + x_d^2} \quad \mbox{and} \quad x \cdot y := x_1 y_1 + \cdots + x_d y_d,
  \\
  \mbox{ for all } x = [x_1 , \dots , x_d], \, y = [y_1 , \dots , y_d] \in \R^d.
\end{gather*}
\end{notn}
Additionally, we note the following elementary fact:
\begin{description}
  \item[\textbf{(Fact 1)}](cf.\cite[Proposition 1.80]{MR1857292}) Let $ m \in \N $ be fixed. If $ \{A_k\}_{k=1}^m \subset \R $ and $ \{ a_n^k \}_{n\in\N}\subset\R $, 
  $ k = 1, \dots , m $ satisfies that:
  \[ \liminf_{ n \rightarrow \infty } a^k_n \geq A_k, \mbox{ for } k = 1, \dots , m, \mbox{ and }\limsup_{ n \rightarrow \infty} \sum_{k=1}^{m} a_n^k \leq \sum_{k=1}^{m} A_k. \]
  Then, $ \lim_{n \rightarrow \infty } a_n^k = A_k $, for $ k = 1 , \dots , m $. 
\end{description}
\begin{notn}[\bf{Abstract functional analysis}]\label{notn2}
For an abstract Banach space $ X $, we denote by $| \cdot |_X$ the norm of $ X $, and denote by $ \langle \cdot , \cdot \rangle_X $ the duality pairing between $ X $ and its dual $ X^* $. In particular, when $ X $ is a Hilbert space, we denote by $( \cdot , \cdot )_X$ its inner product. 

For Banach spaces $ X_1 , \dots ,X_d $ with $ 1 < d \in \N$, let $ X_1 \times \dots \times X_d $ be the product Banach space which has the norm 
\[ | \cdot |_{ X_1\times \dots \times X_d } : = | \cdot |_{X_1} + \dots + | \cdot |_{X_d} .\] 
However, when all $ X_1 , \dots ,X_d $ are Hilbert spaces, $ X_1 \times \dots \times X_d $ denotes the product Hilbert space which has the inner product 
\[ ( \cdot, \cdot )_{ X_1 \times \dots \times X_d} := ( \cdot , \cdot )_{X_1} + \dots + ( \cdot , \cdot )_{X_d} ,\]
and the norm 
\[ | \cdot |_{ X_1 \times \dots \times X_d} := \left( | \cdot |_{X_1}^2 + \dots + | \cdot |_{X_d}^2 \right)^\frac{1}{2} .\] 
In particular, when all $ X_1 , \dots ,X_d $ coincide with a Banach space $ Y $, the product space $ X_1 \times \dots \times X_d $ is simply denoted by $[Y]^d$.
\end{notn}

\begin{notn}[\bf{Convex analysis}]\label{notn3}
For any proper lower semi-continuous (denoted by l.s.c., hereafter) and convex function $ \Psi : X \rightarrow (-\infty,\infty]$ on a Hilbert space $ X $, we denote by $ D( \Psi ) $ the effective domain of $ \Psi $, and denote by $ \partial \Psi $ the subdifferential of $ \Psi $. The subdifferential $ \partial \Psi $ is a set-valued map corresponding to a weak differential of $ \Psi $, and it is known as a maximal monotone graph in the product space $ X \times X $. More precisely, for each $ x_0 \in X $, the value $ \partial \Psi( x_0 ) $ is defined as the set of all elements $ x_0^* \in X $ that satisfy the variational inequality
\[( x_0^*, x - x_0 )_X \leq \Psi (x) - \Psi (x_0), \mbox{ for any } x \in D ( \Psi ), \]
and the set $ D ( \partial \Psi ) := \{ x \in X \,|\, \partial \Psi (x) \neq \emptyset \}$ is called the domain of $ \partial \Psi $. We often use the notation $`` [x_0, x_0^*] \in \partial \Psi $ in $ X \times X "$ to mean that $`` x_0^* \in \partial \Psi ( x_0 ) $ in $ X $ for $ x_0 \in D ( \partial \Psi ) "$, by identifying the operator $ \partial \Psi $ with its graph in $ X \times X $.
\end{notn}
\begin{ex}\label{ex1}
  Let $\gamma: \R^2 \longrightarrow [0,+\infty)$ be a convex function such that it belongs to $C^{0,1}(\R^2)$, $\gamma(-w)= \gamma(w) \mbox{ for any $ w \in \R^2 $}$, and the origin $ 0 \in \R^2 $ is the unique minimizer of $ \gamma $. Also, the function $\{\gamma_\varepsilon\}_{\varepsilon\geq0}$ is defined as follows: 
  \begin{equation*}
    \gamma_\varepsilon :=\left\{
      \begin{aligned}
        &\gamma, &&\mbox{ if }\varepsilon=0,
        \\
        &\rho_\varepsilon * \gamma, &&\mbox{ otherwise},
      \end{aligned}\right.
  \end{equation*}
  where $\rho_\varepsilon$ is the standard mollifier. Then, the following two items hold.
  \begin{description}
    \item[(I)] Let $\{\Phi_\varepsilon\}_{\varepsilon\geq0}$ be a sequence of functionals on $[L^2(\Omega)]^2$, defined as:
  \begin{align*}
    &\Phi_\varepsilon:\bm{w}\in[L^2(\Omega)]^2\mapsto\Phi_\varepsilon(\bm{w}):=\int_\Omega\gamma_\varepsilon(\bm{w})\,dx\in[0,\infty).
  \end{align*}
  Then, for every $\varepsilon\in[0,\infty)$, $\Phi_\varepsilon$ is the proper, l.s.c., and convex function, such that 
  \[D(\Phi_\varepsilon)=D(\partial\Phi_\varepsilon)=[L^2(\Omega)]^2,\]
  and 
  \begin{equation*}
    \partial\Phi_\varepsilon(\bm{w}):=
    \left\{
      \begin{aligned}
        &\{\nabla\gamma_\varepsilon(\bm{w})\}, \mbox{ if }\varepsilon>0,
        \\
        &\{\bm{w}^*\in[L^2(\Omega)]^2~|~\bm{w}^*\in\partial\gamma(\bm{w})\mbox{ a.e. in } \Omega,\}, \mbox{ if }\varepsilon=0,
      \end{aligned}
    \right.
  \end{equation*}
  \[\mbox{ in }[L^2(\Omega)]^2, \mbox{ for any }\bm{w}\in[L^2(\Omega)]^2.\]
    \item[(II)] Let any open interval $I\subset(0,T)$, and let $\{\widehat{\Phi}_\varepsilon^I\}_{\varepsilon\geq0}$ be a sequence of functionals on $L^2(I;[L^2(\Omega)]^2)\,(=[L^2(I;L^2(\Omega))]^2)$, defined as:
  \begin{align*}
    &\widehat{\Phi}_\varepsilon^I:\bm{w}\in L^2(I;[L^2(\Omega)]^2)\mapsto\widehat{\Phi}_\varepsilon^I(\bm{w}):=\int_I\Phi_\varepsilon(\bm{w}(t))\,dt\in[0,\infty).
  \end{align*}
  Then, for every $\varepsilon\in[0,\infty)$, $\widehat{\Phi}^I_\varepsilon$ is the proper, l.s.c., and convex function, such that 
  \[D(\widehat{\Phi}_\varepsilon^I)=D(\partial\widehat{\Phi}_\varepsilon^I)=L^2(I;[L^2(\Omega)]^2),\]
  and 
  \begin{align*}
    \partial\widehat{\Phi}_\varepsilon^I(\bm{w})
    &=\{\tilde{\bm{w}}^*\in L^2(I;[L^2(\Omega)]^2)~|~\tilde{\bm{w}}^*(t)\in \partial \Phi_\varepsilon(\bm{w}(t))\mbox{ in }[L^2(\Omega)]^2, \mbox{ a.e. }t\in I\}
    \\
    &=
    \left\{
      \begin{aligned}
        &\{\nabla\gamma_\varepsilon(\bm{w})\}, \mbox{ if }\varepsilon>0,
        \\
        &\{\bm{w}^*\in L^2(I;[L^2(\Omega)]^2)~|~\bm{w}^*\in\partial\gamma(\bm{w})\mbox{ a.e. in } I\times \Omega,\}, \mbox{ if }\varepsilon=0,
      \end{aligned}
    \right.
  \end{align*}
  \[\mbox{ in }L^2(I;[L^2(\Omega)]^2), \mbox{ for any }\bm{w}\in L^2(I;[L^2(\Omega)]^2).\]
  \end{description}
\end{ex}
\begin{definition}[\bf{Mosco-convergence}: cf.\cite{MR0298508}]\label{dfnmosco}
    Let $ X $ be a Hilbert space. Let $ \Psi : X \rightarrow ( -\infty , \infty ] $ be a proper, l.s.c., and convex function, and let $ \{ \Psi_n \}_{ n \in \N } $ be a sequence of proper, l.s.c., and convex functions $ \Psi_n : X \rightarrow ( -\infty , \infty ] $, $ n \in \N $. Then, we say that $ \Psi_n \to  \Psi $ on $ X $ in the sense of Mosco, iff. the following two conditions are fulfilled:
  \begin{description}
    \item[(M1) (Optimality)] For any $w_0 \in D ( \Psi )$, there exists a sequence $ \{w_n\}_{ n \in \N } \subset X $ such that $ w_n \rightarrow w_0 $ in $ X $ and $ \Psi_n ( w_n ) \rightarrow \Psi ( w_0 ) $ as $ n \rightarrow \infty $, 
    \item[(M2) (Lower-bound)] $\liminf_{ n \rightarrow \infty } \Psi_n ( w_n ) \geq \Psi ( w_0 )$ if $w_0 \in X, \{ w_n \}_{ n \in \N } \subset X $, and $w_n \rightarrow w_0 $ weakly in $ X $ as $ n \rightarrow \infty $.
  \end{description}
\end{definition}
\begin{definition}[\bf{$ \Gamma $-convergence}; cf.\cite{MR1201152}]\label{dfngamma}
  Let $ X $ be a Hilbert space. Let $ \Psi : X \rightarrow ( -\infty , \infty ] $ be a proper and l.s.c. function, and let $ \{ \Psi_n \}_{ n \in \N } $ be a sequence of proper and l.s.c. functions $ \Psi_n : X \rightarrow ( -\infty , \infty ] $, $ n \in \N $. Then, we say that {$ \Psi_n \to \Psi $} on $ X $ in the sense of $ \Gamma $-convergence, iff. the following two conditions are fulfilled:
  \begin{description}
      \item[{\boldmath ($\Gamma$1) (Optimality)}] For any $w_0 \in D ( \Psi )$, there exists a sequence $ \{w_n\}_{ n \in \N } \subset X $ such that $ w_n \rightarrow w_0 $ in $ X $ and $ \Psi_n ( w_n ) \rightarrow \Psi ( w_0 ) $ as $ n \rightarrow \infty $, 
      \item[{\boldmath ($\Gamma$2) (Lower-bound)}] $\liminf_{ n \rightarrow \infty } \Psi_n ( w_n ) \geq \Psi ( w_0 )$ if $w_0 \in X, \{ w_n \}_{ n \in \N } \subset X $, and $w_n \rightarrow w_0 $ in $ X $ as $ n \rightarrow \infty $.
  \end{description}
\end{definition}

\begin{rem}\label{rem2}
  We note that under the condition of convexity of functionals, Mosco convergence implies $ \Gamma $-convergence, i.e., the  $ \Gamma $-convergence of convex functions can be regarded as a weak version of Mosco convergence. Furthermore,  in the $ \Gamma $-convergence of convex functions, we can see the following:
\begin{description}
  \item[(Fact 2)] (cf.\cite[Theorem 3.66]{MR0773850} and \cite[Chapter 2]{Kenmochi81}) Let $ X $ be a Hilbert space. Let $ \Psi : X \rightarrow ( -\infty , \infty ] $ and $ \Psi_n : X \rightarrow ( -\infty , \infty ] $, $ n \in \N $, be proper, l.s.c., and convex functions on a Hilbert space $ X $ such that $ \Psi_n \rightarrow \Psi $ on $ X $, in the sense of $ \Gamma $-convergence, as $ n \rightarrow \infty $. Let us assume that
  \begin{equation*}
    \left\{
    \begin{aligned}
      &[z, z^*] \in X \times X ,~ [z_n, z_n^*] \in \partial \Psi_n \mbox{ in } X \times X,~n \in \N,
      \\
      &z_n \rightarrow z^* \mbox{ in } X \mbox{ and } z_n^* \rightarrow z^* \mbox{ weakly in } X \mbox{ as } n \rightarrow \infty.  
    \end{aligned}
    \right.
  \end{equation*}
  Then, it holds that:
  \[ [ z , z^* ] \in \partial \Psi \mbox{ in } X \times X , \mbox{ and } \Psi_n ( z_n ) \rightarrow \Psi ( z ) \mbox{ as } n \rightarrow \infty. \]
  \item[(Fact 3)](cf.\cite[Lemma 4.1]{MR3661429} and \cite[Appendix]{MR2096945}) Let $ X $ be a Hilbert space, $ d \in \N $ be dimension constant, and $ A \subset \R^d $ be a bounded open set. Let $ \Psi : X \rightarrow ( -\infty , \infty ] $ and $ \Psi_n : X \rightarrow ( -\infty , \infty ] $, $ n \in \N $, be proper, l.s.c., and convex functions on a Hilbert space $ X $ such that $ \Psi_n \rightarrow \Psi $ on $ X $, in the sense of $ \Gamma $-convergence, as $ n \rightarrow \infty $. Then, a sequence $ \{ \widehat{ \Psi }_n^A \}_{ n \in \N } $ of proper, l.s.c., and convex functions on $ L^2 ( A ; X ) $, defined as:
  \begin{equation*}
    z \in L^2 ( A ; X ) \mapsto \widehat{ \Psi }^A_n ( z ) : = \left\{
      \begin{aligned}
        & \int_A \Psi _n ( z ( t ) ) \,dt, \mbox{ if } \Psi_n ( z ) \in L^1 ( A ), 
        \\
        & \infty, \mbox{ otherwise}, 
      \end{aligned}
    \right.\mbox{ for }n \in \N;
  \end{equation*}
  converges to a proper, l.s.c., and convex function $ \widehat{ \Psi }^A $ on $ L^2 ( A ; X ) $, defined as:
  \begin{equation*}
    z \in L^2 ( A ; X ) \mapsto \widehat{ \Psi }^A ( z ) : = \left\{
      \begin{aligned}
        & \int_A \Psi ( z ( t ) ) \,dt, \mbox{ if } \Psi ( z ) \in L^1 ( A ), 
        \\
        & \infty, \mbox{ otherwise}, 
      \end{aligned}
    \right.
  \end{equation*}
  on $ L^2 ( A ; X ) $, in the sense of $ \Gamma $-convergence, as $ n \rightarrow \infty $.
\end{description}
\end{rem}
\begin{ex}[Examples of Mosco-convergence]\label{ex2}
  Let $\varepsilon_0\geq0$ be arbitrary fixed constant, and let $\gamma$ and $\{\gamma_\varepsilon\}_{\varepsilon\geq0}$ be as in Example \ref{ex1}, respectively. Then, the following three items hold.
  \begin{description}
    \item[(I)] $\gamma_\varepsilon\rightarrow\gamma_{\varepsilon_0}$ on $\R^2$, in the sense of Mosco, as $\varepsilon\rightarrow\varepsilon_0$.
    \item[(II)] Let $\{\Phi_\varepsilon\}_{\varepsilon\geq0}$ be the sequence of proper, l.s.c., and convex functions on $[L^2(\Omega)]^2$, as in Example \ref{ex1}(I). Then,
    \begin{align}
    \Phi_\varepsilon\rightarrow\Phi_{\varepsilon_0} \mbox{ on }[L^2(\Omega)]^2 \mbox{ in the sense of Mosco, as }\varepsilon\rightarrow\varepsilon_0.
    \end{align} 
    \item[(III)] Let $I\subset(0,T)$ be an open interval, and let $\{\widehat{\Phi}_{\varepsilon}^I\}_{\varepsilon\geq0}$ be the sequence of proper, l.s.c., and convex functions on $L^2(I;[L^2(\Omega)]^2)$, as in Example \ref{ex1}(II). Then, 
    \begin{align}
    \widehat{\Phi}^I_\varepsilon\rightarrow\widehat{\Phi}^I_{\varepsilon_0} \mbox{ on }L^2(I;[L^2(\Omega)]^2), \mbox{ in the sense of Mosco, as }\varepsilon\rightarrow\varepsilon_0.
  \end{align}
  \end{description}
\end{ex}
\begin{notn}\label{deftseq}
  Let $ \tau >0 $ be a constant of time-step size, and $\{t_i\}_{i=0}^\infty$ be a time-sequence defined as 
  \[t_i := i\tau, \mbox{ for any } i = 0, 1, 2, \dots . \]
  Let $ X $ be a Banach space. For any sequence $ \{ [t_i , u_i] \}_{i=0}^\infty \subset [0,\infty) \times X $, we define three types of interpolations: $ [ \overline{u} ]_\tau \in L^\infty_{\mathrm{loc}}([0,\infty);X) $, $ [ \underline{u} ]_\tau \in L^\infty_{\mathrm{loc}}([0,\infty);X) $, and $ [ u ]_\tau \in W^{1,2}_{\mathrm{loc}}([0,\infty);X) $, as follows:
\begin{equation*}
  \left\{
  \begin{aligned}
    &[\overline{u}]_\tau(t):=\chi_{(-\infty,0]}u_{0}+\sum_{i=1}^{\infty}\chi_{(t_{i-1},t_i]}(t)u_{i},
    \\
    &[\underline{u}]_\tau(t):=\sum_{i=0}^{\infty}\chi_{(t_i,t_{i+1}]}(t)u_{i},\hspace*{25ex}\mbox{in }X,\mbox{ for any } t\geq0,
    \\
    &[u]_\tau(t):=\sum_{i=1}^{\infty}\chi_{(t_{i-1},t_i]}(t)\biggl(\frac{t-t_{i-1}}{\tau}u_{i}+
    \frac{t_i-t}{\tau}u_{i-1}\biggr),
  \end{aligned}
  \right.
\end{equation*}
where $ \chi_E : \R \rightarrow \{0,1\} $ denotes the characteristic function of the set $ E \subset \R $.
\end{notn}

\section{Main results}
\label{sec:main}

In this paper, the Main Theorems are discussed under the following assumptions.

\begin{itemize}
  \item[(A0)] $p > 2$, $\kappa > 0$, $\nu > 0$, $\mu > 0$, and $\lambda > 0 $ are fixed constants.
    \vspace{1ex}
  \item[(A1)] $ u_{\mathrm{org}} \in L^2 ( \Omega ) $ is a fixed function such that $ 0 \leq u_\mathrm{org} \leq 1 $ a.e. in $ \Omega $. 
  \item[(A2)] $\gamma: \R^2 \longrightarrow [0,+\infty)$ is a fixed convex function such that:
      \begin{gather*}    
          \gamma(-w)= \gamma(w) \mbox{ for any $ w \in \R^2 $, }\operatorname*{arg\,min}_{w \in \mathbb{R}^2} \gamma(w) = \{0\}, \mbox{ and } \nabla \gamma \in L^\infty(\R^2; \R^2),
      \end{gather*}
        i.e., $ \gamma $ is Lipschitz continuous, and the origin $ 0 \in \R^2 $ is the unique minimizer of $ \gamma $.  
  \item[(A3)] $[u_0, \alpha_0]$ is a fixed initial data such that  $[u_0,\alpha_0]\in W^{1,p}_0(\Omega)\times H^1_0(\Omega)$, and $0\leq u_0 \leq1$ a.e. in $\Omega$.
\end{itemize}

Next, let us give the definition of the solution to the system (S).
\medskip

\begin{definition}
  A pair of functions $[u,\alpha]\in [L^2(0,T;L^2(\Omega))]^2$ is called a solution to the system (S) if and only if $[u,\alpha]$ fulfills the following conditions:
  \begin{itemize}
      \item[(S0)] $[u, \alpha] \in [W^{1,2}(0,T;H^1_0(\Omega))\cap L^\infty(0,T;W^{1,p}_0(\Omega))] \times W^{1,2}(0,T; H^1_0(\Omega))$, and {$0 \leq u\leq 1 $ in $\overline{Q}$}.
      \vspace{1ex}
    \item[(S1)] $u$ solves the following variational inequality:
      \begin{gather*}
        (\partial_tu(t),u(t)-\psi)_{L^2(\Omega)}+\lambda(u(t)-u_\mathrm{org},u(t)-\psi)_{L^2(\Omega)}
        \\
        +\mu(\nabla\partial_tu(t),\nabla(u(t)-\psi))_{[L^2(\Omega)]^2}
        +\nu\int_{\Omega}|\nabla u(t)|^{p-2}\nabla u(t)\cdot\nabla(u(t)-\psi)\,dx
        \\
          +\int_{\Omega}\gamma(R(\alpha(t))\nabla u(t))\,dx
        \leq\int_{\Omega}\gamma (R(\alpha(t))\nabla\psi)\,dx,
        \\
        \mbox{ for any }\psi\in W^{1,p}_0(\Omega),\mbox{ and a.e. }t\in(0,T).  
      \end{gather*}
      \vspace{-3ex}
  \item[(S2)] There exists a vector-value function $ \bm{w}^* \in L^2(0, T; [L^2(\Omega)]^2) $ such that:
      \begin{gather*}
        \bm{w}^*\in\partial\gamma(R(\alpha)\nabla u)\mbox{ in }\R^2, ~\mbox{ a.e. in $ Q $,}
      \end{gather*} 
      and $\alpha$ solves the following variational identity: 
      \begin{gather*}
        (\partial_t\alpha(t),\varphi)_{L^2(\Omega)}+(\nabla\alpha(t),\nabla\varphi)_{[L^2(\Omega)]^2}
        +\kappa(\nabla\partial_t\alpha(t),\nabla\varphi)_{[L^2(\Omega)]^2}
        \\
          +\int_{\Omega}\bm{w}^*(t)\cdot R\left(\alpha(t){\ts +\frac{\pi}{2} }\right)\nabla u(t)\varphi \,dx=0,
        \mbox{ for any }\varphi\in H^1_0(\Omega),\mbox{ and a.e. }t\in(0,T).
      \end{gather*} 
  \item[(S3)] $[u(0),\alpha(0)] = [u_0, \alpha_0]$ in $ [L^2(\Omega)]^2 $.
  \end{itemize}
\end{definition}
\pagebreak

Now, the Main Theorems are stated as follows.

\begin{mainThm}(Existence of solution with energy-dissipation)\label{mainThm1}
  Under the assumptions (A0)--(A3), the system (S) admits a solution $[u,\alpha]$. 
  Moreover, the solution $[u,\alpha]$ fulfills the following energy-inequality:
  \begin{gather}
    \frac{1}{2}\int_{s}^{t}|\partial_t\alpha(\sigma)|^2_{L^2(\Omega)}\,d\sigma
    +\frac{\kappa}{2}\int_{s}^{t}|\nabla\partial_t\alpha(\sigma)|^2_{[L^2(\Omega)]^2}\,d\sigma+\int_{s}^{t}|\partial_tu(\sigma)|^2_{L^2(\Omega)}\,d\sigma
    \\
      +{\mu}\int_{s}^{t}|\nabla\partial_tu(\sigma)|^2_{[L^2(\Omega)]^2}\,d\sigma+E({u}(t),{\alpha}(t))\leq E({u}(s),{\alpha}(s)),
      \\
      \mbox{for a.e. $ s \in [0, T) $ including $ s = 0 $, and any } t \in [s, T].
    \label{ene-inq1}
  \end{gather}
\end{mainThm}
\begin{mainThm}(Uniqueness and continuous dependence of solution)\label{mainThm2}
    Under the assumption (A0)--(A3), let $ [u_{0, k}, \alpha_{0, k}] \in W_0^{1, p}(\Omega) \times H_0^1(\Omega) $, $ k = 1, 2 $, be two initial data, and let $ [u_k, \alpha_k] \in [L^2(0, T; L^2(\Omega))]^2 $, $ k = 1, 2 $, be two solutions to the system (S) when $ [u_0, \alpha_0] = [u_{0, k}, \alpha_{0, k}] $, $ k = 1, 2 $. Additionally, let us assume that $\gamma \in C^{1, 1}(\R^2) \cap C^2(\R^2) $.  Then, for a function of time, defined as:
    \begin{gather}\label{defOfJ}
        J(t) := |(u_1 -u_2)(t)|^2_{L^2(\Omega)} +\mu |\nabla (u_1 -u_2)|_{[L^2(\Omega)]^2} 
        \nonumber
        \\
        +|(\alpha_1 -\alpha_2)(t)|^2_{L^2(\Omega)} +\kappa |\nabla (\alpha_1 -\alpha_2)(t)|_{[L^2(\Omega)]^2}, \mbox{ for all } t \in [0,T],
    \end{gather}
    there exists a positive constant $ C_* $, such that
    \begin{gather}\label{cncl_J}
          J(t)\leq\exp(C_*T(1+|u_1|_{L^\infty(0,T;W^{1,p}_0(\Omega))})^2)J(0), \mbox{ for all $ t \in [0, T] $.}
    \end{gather}
    {
        Moreover, the energy-inequality \eqref{ene-inq1} holds for all $ 0 \leq s \leq t \leq T $. 
    }
\end{mainThm}

\section{Time-discretization scheme}
In this section, we focus on the time-discretization scheme of our system (S). Let $ m \in \N $ be the division number of the time-interval $ (0, T) $. Let $\tau := \frac{T}{m}$ be the time-step size, and let $\{t_i\}_{i=1}^m$ be a time-sequence defined as $t_i := i\tau$, $i = 1, 2, \dots, m$. Moreover, throughout this section, we suppose that the function ${\gamma}$ belongs to $C^{1,1}(\R^2) \cap C^2(\R^2)$, by taking suitable regularization if necessary.  

Based on these, we set up the time-discretization scheme (AP)$_\tau$ as an approximating problem of (S):
\\

(AP)$_\tau$: \, To find $\{[u_{i}, \alpha_{i}]\}_{i=1}^m \subset W^{1,p}_0(\Omega) \times H^1_0(\Omega) $ satisfying: 
  \begin{align*}
    &\frac{\alpha_{i}-\alpha_{i-1}}{\tau}
    -\Delta_0\alpha_{i}
    -\frac{\kappa}{\tau}\Delta_0(\alpha_{i}-\alpha_{i-1})
    \\
    &\quad +\nabla\gamma(R(\alpha_{i-1})\nabla u_{i-1})\cdot R\left(\alpha_{i-1}+{\ts \frac{\pi}{2}}\right)
    \nabla u_{i-1}=0~\mathrm{in}~L^2(\Omega),\nonumber
    \\
    &\frac{u_{i}-u_{i-1}}{\tau}
    -\mathrm{div}\Bigl( {^\top} R(\alpha_{i})\nabla{\gamma}(R(\alpha_{i})\nabla u_{i})
    +\nu|\nabla u_{i}|^{p-2}\nabla u_{i}
    \\
    &\quad +\frac{\mu}{\tau}\nabla(u_{i}-u_{i-1})\Bigr)
    +\lambda(u_{i}-u_{\mathrm{org}})
    =0~\mathrm{in}~L^2(\Omega),\nonumber
      ~~\mbox{for }i=1,2,\dots,m,
    \end{align*}
where $[u_{0},\alpha_{0}] \in W^{1,p}_0(\Omega) \times H^1_0(\Omega) $ is as in (A3), and $\Delta_0$ is the Laplacian operator subject to the zero-Dirichlet boundary condition. 

The solution to (AP)$_\tau$ is given as follows.
\begin{definition}
  A sequence of pair of functions $\{[u_{i}, \alpha_{i}]\}_{i=1}^m$ is called a solution to (AP)$_\tau$ if $\{[u_{i}, \alpha_{i}]\}_{i=1}^m \subset W^{1,p}_0(\Omega) \times H^1_0(\Omega)$, and $[u_{i}, \alpha_{i}]$ fulfills the following items for any $i = 1,2,\dots,m$:
  \\
  (AP1) $\alpha_{i}$ solves the following variational identity:
  \begin{align*}
    & \frac{1}{\tau}(\alpha_{i}-\alpha_{i-1},\varphi)_{L^2(\Omega)}+
    (\nabla\alpha_{i},\nabla\varphi)_{[L^2(\Omega)]^2}
    +\frac{\kappa}{\tau}(\nabla(\alpha_{i}-\alpha_{i-1}),\nabla\varphi)_{[L^2(\Omega)]^2}
    \\
    &\quad +(\nabla{\gamma}(R(\alpha_{i-1})\nabla u_{i-1}) \cdot 
    R\left(\alpha_{i-1}+{\ts \frac{\pi}{2}}\right)\nabla u_{i-1},\varphi)_{L^2(\Omega)}
    =0,~~~\mbox{ for any }\varphi\in H^1_0(\Omega).
  \end{align*}
  (AP2) $u_{i}$ solves the following variational identity:
  \begin{align*}
    & \frac{1}{\tau}(u_{i}-u_{i-1},\psi)_{L^2(\Omega)}+(\nabla\gamma(R(\alpha_{i})\nabla u_{i}),R(\alpha_{i})\nabla\psi)_{[L^2(\Omega)]^2}
    \\
    &\quad +\nu\int_\Omega|\nabla u_{i}|^{p-2}\nabla u_{i}\cdot\nabla\psi\,dx
    +\frac{\mu}{\tau}(\nabla(u_{i}-u_{i-1}),\nabla\psi)_{[L^2(\Omega)]^2}
    \\
    &\quad +\lambda(u_{i}-u_{\mathrm{org}},\psi)_{L^2(\Omega)}=0,~~~\mbox{ for any }\psi\in W^{1,p}_0(\Omega).
  \end{align*}
\end{definition}
\begin{theorem}(Existence, uniqueness of solution with energy-dissipation)\label{003Thm1}
  There exists a sufficiently small constant $\tau_*:=\tau_*(|\nabla\gamma|_{W^{1,\infty}(\R^2;\R^2)})\in (0,1)$, possibly dependent on $\gamma$, such that for any $ \tau \in (0,\tau_*) $, (AP)$_\tau$ admits a unique solution $\{[u_{i}, \alpha_{i}]\}_{i=1}^m$. Moreover, the solution $\{[u_{i}, \alpha_{i}]\}_{i=1}^m$ fulfills the following energy-inequality:
  \begin{align}
    &\frac{1}{2\tau}|\alpha_{i}-\alpha_{i-1}|^2_{L^2(\Omega)} +\frac{\kappa}{2\tau}|\nabla(\alpha_{i}-\alpha_{i-1})|^2_{[L^2(\Omega)]^2} +\frac{1}{\tau}|u_{i}-u_{i-1}|^2_{L^2(\Omega)}
    \label{f-ene0}
    \\
    &+\frac{\mu}{\tau}|\nabla(u_{i}-u_{i-1})|^2_{[L^2(\Omega)]^2} +E(u_{i},\alpha_{i}) \leq E(u_{i-1},\alpha_{i-1}),\mbox{ for any } i=1,2,\dots,m.
    \nonumber
  \end{align}
\end{theorem}

We prepare some lemmas for the proof of Theorem \ref{003Thm1}. 
\begin{lemma}\label{003lemma1}{
  If $ \alpha_0^\dagger \in H^1_0(\Omega) $ and $ \tilde{u} \in H^{1} (\Omega) $, then there exist a unique solution $\alpha\in H^1_0(\Omega)$ of the following elliptic boundary value problem \eqref{AP1} and \eqref{AP1sol} holds:
  \begin{gather}
    \frac{1}{\tau}(\alpha-\alpha_0^\dagger)+\Delta_0(\alpha+\frac{\kappa}{\tau}(\alpha-\alpha_0^\dagger))
    +\nabla\gamma(R(\alpha_0^\dagger)\nabla \tilde{u})\cdot
    R(\alpha_0^\dagger+{\ts \frac{\pi}{2}})\nabla\tilde{u}=0~\mathrm{in}~\Omega,\label{AP1}
  \end{gather}
and
  \begin{gather}
    \frac{1}{\tau}(\alpha-\alpha_0^\dagger,\varphi)_{L^2(\Omega)}
    +(\nabla\alpha,\nabla\varphi)_{[L^2(\Omega)]^2}
    +\frac{\kappa}{\tau}(\nabla(\alpha-\alpha_0^\dagger),\nabla\varphi)_{[L^2(\Omega)]^2}\nonumber
    \\
    +(\nabla\gamma(R(\alpha_0^\dagger)\nabla \tilde{u})\cdot R(\alpha_0^\dagger+
    {\ts \frac{\pi}{2}})\nabla\tilde{u},\varphi)_{L^2(\Omega)}=0,\mbox{ for any }\varphi\in H^1_0(\Omega).
    \label{AP1sol}
  \end{gather}}
  \begin{proof}
    Let us consider a proper, l.s.c., strictly convex, and coercive function $\Psi: L^2(\Omega) \rightarrow (-\infty,\infty]$ defined as follows: 
    \begin{equation*}
      \Psi : z \in L^2(\Omega) \mapsto \Psi(z) : = \left\{
      \begin{aligned}
        &\frac{1}{2\tau}\int_\Omega|z-\alpha_0^\dagger|^2\,dx+\frac{1}{2}\int_\Omega|\nabla z|^2\,dx+\frac{\kappa}{2\tau}\int_\Omega|\nabla (z-\alpha_0^\dagger)|^2\,dx
        \\
        &\quad+\int_\Omega\nabla\gamma(R(\alpha_0^\dagger)\nabla\tilde{u})\cdot R(\alpha_0^\dagger+{\ts \frac{\pi}{2}})\nabla\tilde{u}z\,dx, \mbox{ if } z \in H^1_0(\Omega),
        \\
        &+\infty,~\mbox{ otherwise}.
      \end{aligned}
      \right.
    \end{equation*}
    Since $\Psi$ is strictly convex, the minimizer of $\Psi$ is unique, and it directly leads to the existence and uniqueness of solution $\alpha\in H^1_0(\Omega)$ to \eqref{AP1}.
  \end{proof}
\end{lemma}
\begin{lemma}\label{003lemma2}{
  If $ u_0^\dagger \in W^{1,p}_0(\Omega) $ and $ \tilde{\alpha} \in L^2(\Omega) $, then there exist a unique solution $u\in W^{1,p}_0(\Omega)$ of the following elliptic boundary value problem \eqref{AP2} and \eqref{AP2sol} holds:
  \begin{align}
    &\frac{1}{\tau}(u-u_0^\dagger)-
    \mathrm{div}
    \biggl(  
      {^\top} R(\tilde{\alpha})\cdot\nabla\gamma(R(\tilde{\alpha})\nabla u)+\nu |\nabla u|^{p-2}\nabla u+\frac{\mu}{\tau}\nabla(u-u_0^\dagger)
    \biggr)\label{AP2}
    \\
    &\qquad+\lambda(u-u_{\mathrm{org}})=0 \mathrm{~in~} \Omega,
    \nonumber
  \end{align}
and 
  \begin{align}
    &\frac{1}{\tau} (u-u_0^\dagger,\psi)_{L^2(\Omega)}+ \nu \int_\Omega |\nabla u|^{p-2}\nabla u \cdot \nabla\psi\,dx+\frac{\mu}{\tau}(\nabla(u-u_0^\dagger),\nabla\psi)_{[L^2(\Omega)]^2}
    \label{AP2sol}
    \\
    &\quad+(\nabla\gamma(R(\tilde{\alpha})\nabla u),R(\tilde{\alpha})\nabla\psi)_{[L^2(\Omega)]^2}+\lambda(u-u_{\mathrm{org}},\psi)_{L^2(\Omega)}=0, \mbox{ for any } \psi \in W^{1,p}_0(\Omega).
    \nonumber
  \end{align}}
  \begin{proof}
    Let us consider a proper, l.s.c., strictly convex, and coercive function $\Phi: L^2(\Omega) \rightarrow (-\infty,\infty]$ defined as follows: 
    \begin{equation*}
      \Phi : z \in L^2(\Omega) \mapsto \Phi(z) : = \left\{
      \begin{aligned}
        &\frac{1}{2\tau}\int_\Omega|z-u_0^\dagger|^2\,dx
        +\frac{\nu}{p}\int_\Omega|\nabla z|^p\,dx
        +\frac{\mu}{2\tau}\int_\Omega|\nabla (z- u_0^\dagger)|^2\,dx
        \\
        &\quad+\int_\Omega\gamma(R(\tilde{\alpha})\nabla z)\,dx
        +\frac{\lambda}{2}\int_\Omega|z-u_{\mathrm{org}}|^2\,dx
        ~\mathrm{if}~z\in W^{1,p}_0(\Omega),
        \\
        &+\infty,~\mathrm{otherwise}.
      \end{aligned}
      \right.
    \end{equation*}
    Since $\Phi$ is strictly convex, the minimizer of $\Phi$ is unique, and it directly leads to the existence and uniqueness of solution $ u \in W^{1,p}_0(\Omega) $ to \eqref{AP2}.
  \end{proof}
\end{lemma}
\begin{proof}[Proof of Theorem \ref{003Thm1}]
    First, for a simplicity, let us set
  \begin{align*}
      &c_{E_0}:=E(u_{0},\alpha_{0}) \in [0, \infty).
  \end{align*}
    
    The existence and uniqueness of the solution to (AP)$_\tau$ is a straightforward result of {Lemmas \ref{003lemma1} and \ref{003lemma2}}. In fact, applying {these lemmas and putting}
  \begin{align*}
    u_0^\dagger:=u_{i-1},~\tilde{u}:=u_{i-1},~\alpha_0^\dagger:=\alpha_{i-1},~\tilde{\alpha}:=\alpha_{i},
  \end{align*}
    we immediately obtain the unique solution $ \{ [u_i, \alpha_i] \}_{i = 1}^m $, inductively.

Next, we will prove the inequality \eqref{f-ene0}. {Let} $ \varphi : = \alpha_{i} - \alpha_{i-1} $ in \eqref{AP1sol} and $ \psi : = u_{i} - u_{i-1} $ in \eqref{AP2sol}. By Young's inequality and the convexity of $ \gamma $, we can see that:
\begin{align}
  & \frac{1}{\tau} | \alpha_{i} - \alpha_{i-1} | ^2 _{L^2(\Omega)}
  +\frac{1}{2} | \nabla \alpha_{i} | ^2 _{[L^2(\Omega)]^2}
  -\frac{1}{2} | \nabla \alpha_{i-1} | ^2 _{[L^2(\Omega)]^2}
  +\frac{\kappa}{\tau} | \nabla ( \alpha_{i} - \alpha_{i-1} ) |^2 _{[L^2(\Omega)]^2}
  \nonumber
  \\
  & \quad +\int_\Omega \nabla \gamma ( R (\alpha_{i-1} ) \nabla u_{i-1} ) 
  \cdot R \left( \alpha_{i-1} +{\ts \frac{\pi}{2}} \right) \nabla u_{i-1} ( \alpha_{i} - \alpha_{i-1} )
  \leq 0,\label{inequ001}
\end{align}
and 
\begin{align}
  & \frac{1}{\tau} | u_{i} - u_{i-1} | ^2 _{L^2(\Omega)} +\frac{\nu}{p} | \nabla u_{i} | ^p _{[L^p(\Omega)]^2} -\frac{\nu}{p} | \nabla u_{i-1} | ^p _{[L^p(\Omega)]^2}
  \nonumber
  \\
  & \quad +\frac{\mu}{\tau} | \nabla ( u_{i} - u_{i-1} ) |^2 _{[L^2(\Omega)]^2} +\frac{\lambda}{2} | u_{i} - u_{\mathrm{org}} | ^2 _{L^2(\Omega)} -\frac{\lambda}{2} | u_{i-1} - u_{\mathrm{org}} | ^2 _{L^2(\Omega)}  
  \label{inequ002}
  \\
  & \quad +\int_\Omega \gamma ( R (\alpha_{i} ) \nabla u_{i} ) \, dx -\int_\Omega \gamma ( R (\alpha_{i} ) \nabla u_{i-1} ) \, dx \leq 0,\mbox{ for any } i = 1, 2, \dots , m.
  \nonumber
\end{align}
    Also, taking the sum of \eqref{inequ001} and \eqref{inequ002}, and applying Taylor's theorem for $ \gamma \in C^{1, 1}(\R^2) \cap C^2(\R^2) $, we have
\begin{align}
  &\frac{1}{\tau}\big(|\alpha_{i}-\alpha_{i-1}|^2_{L^2(\Omega)} +\kappa|\nabla(\alpha_{i}-\alpha_{i-1})|^2_{[L^2(\Omega)]^2} +|u_{i}- u_{i-1}|^2_{L^2(\Omega)}
  \nonumber
  \\
  &\quad+\mu|\nabla(u_{i}-u_{i-1})|^2_{[L^2(\Omega)]^2}\big)+E(u_{i},\alpha_{i})-E(u_{i-1},\alpha_{i-1})
  \nonumber
  \\
  &\quad\leq\int_\Omega\Bigl[\gamma(R(\alpha_{i})\nabla u_{i-1}) -\gamma(R(\alpha_{i-1})\nabla u_{i-1}) 
  \nonumber
  \\
  &\quad\quad-\nabla\gamma(R(\alpha_{i-1})\nabla u_{i-1})\cdot R\left(\alpha_{i-1}+{\ts\frac{\pi}{2}}\right)\nabla u_{i-1}(\alpha_{i}-\alpha_{i-1})\Bigr]\,dx
  \label{inequ003}
  \\
  &\quad\leq \int_\Omega \int_0^1(1-\sigma)\frac{\partial^2}{\partial \alpha^2} \Big[ \gamma ( R (\alpha_{i-1}+\sigma(\alpha_{i}-\alpha_{i-1}) ) \nabla u_{i-1} ) \Big] ( \alpha_{i} - \alpha_{i-1} )^2 \,d\sigma dx.\nonumber
\end{align}
Here, by using the chain-rule, we easily check that
\begin{align}\label{003gamma2}
    & \left| \frac{\partial^2}{\partial \alpha^2}\Big[ \gamma ( R (\alpha ) \bm{w} ) \Big] \right| \leq 4 |\nabla^2 \gamma|_{L^\infty(\R^2; \R^{2 \times 2})} |\bm{w}|^2 +2 |\nabla \gamma|_{L^\infty(\R^2; \R^2)}|\bm{w}|
    \\
    & \hspace{0ex} \leq 2 \bigl( 1 +|\nabla \gamma|_{W^{1, \infty}(\R^2; \R^2)} \bigr)^2 (1 +|\bm{w}|^2),
    \mbox{ for all } [\alpha, \bm{w}] \in \R \times \R^2.
\end{align}
So, as a consequence of \eqref{inequ003} and \eqref{003gamma2}, it is deduced that
\begin{align}
  &\frac{1}{\tau} 
  \Bigl( 
  | \alpha_{i} - \alpha_{i-1} | ^2 _{L^2(\Omega)} 
  + \kappa | \nabla ( \alpha_{i} - \alpha_{i-1} ) |^2 _{[L^2(\Omega)]^2} 
  + | u_{i} - u_{i-1} | ^2 _{L^2(\Omega)}
  \nonumber
  \\
  &\quad+ \mu | \nabla ( u_{i} - u_{i-1} ) |^2 _{[L^2(\Omega)]^2}
  \Bigr)
  + E ( u_{i} , \alpha_{i} )
  - E ( u_{i-1} , \alpha_{i-1} )
  \label{inequ010}
  \\
    &\quad\leq 2 \bigl(1 +|\nabla \gamma|_{W^{1, \infty}(\R^2; \R^2)} \bigr)^2 \int_\Omega (1 +|\nabla u_{i-1}|^2)(\alpha_{i} - \alpha_{i-1} )^2 \,dx.\nonumber
\end{align}
Furthermore, owing to the two-dimensional embedding
\begin{equation}\label{2D-emb}
    H^1(\Omega) \subset L^{\frac{2p}{p -2}}(\Omega) \mbox{ for $ p > 2 $,}
\end{equation}
the principal integral of the right-hand side of \eqref{inequ010} will be estimated as follows:
\begin{align}
  &\int_\Omega (1+|\nabla u_{i-1}|^2)(\alpha_{i} - \alpha_{i-1} )^2 \,dx 
    \nonumber
    \\
    & \leq (|\nabla u_{i-1}|^2_{[L^p(\Omega)]^2}|\alpha_{i}-\alpha_{i-1}|^2_{L^\frac{2p}{p-2}(\Omega)}+|\alpha_{i}-\alpha_{i-1}|^2_{L^2(\Omega)})
    \nonumber
  \\
  &\leq (1 +(C_{H^1}^{L^\frac{2p}{p-2}})^2)(1 +|\nabla u_{i-1}|^2 _{[L^p(\Omega)]^2})| \alpha_{i}-\alpha_{i-1}|^2_{H^1(\Omega)}
    \nonumber
  \\
  &\leq {C_*}(1 +E(u_{i-1},\alpha_{i-1})^\frac{2}{p})(|\alpha_{i}-\alpha_{i-1}|^2_{L^2(\Omega)}+\kappa|\nabla(\alpha_{i}-\alpha_{i-1})|^2_{[L^2(\Omega)]^2}){,}
    \label{ken01}
\end{align}
where $ C_* $ is a positive constant given as:
\begin{gather*}
    C_* := \frac{(1 +( C_{H^1}^{L^\frac{2p}{p-2}} )^2)(1 +( \frac{p}{\nu} ) ^\frac{2}{p}) }{1\land \kappa},
    \end{gather*}
    {with the constant $ C_{H^1}^{L^{\frac{2p}{p -2}}} $ of the embedding \eqref{2D-emb}.}
    \\
On account of \eqref{inequ003}, \eqref{inequ010}, and \eqref{ken01}, we will infer that
\begin{align}
  &\frac{1}{\tau} \bigl( | \alpha_{i} - \alpha_{i-1} | ^2 _{L^2(\Omega)}  + \kappa | \nabla ( \alpha_{i} - \alpha_{i-1} ) |^2 _{[L^2(\Omega)]^2} + | u_{i} - u_{i-1} | ^2 _{L^2(\Omega)}
  \nonumber
  \\
  &\quad+ \mu | \nabla ( u_{i} - u_{i-1} ) |^2 _{[L^2(\Omega)]^2}\bigr)+ E ( u_{i} , \alpha_{i} )- E ( u_{i-1} , \alpha_{i-1} )
  \nonumber
  \\
    &\leq 2 {C_*}(1 +|\nabla \gamma|_{W^{1, \infty}(\R^2; \R^2)})^2(1+E(u_{i-1},\alpha_{i-1})^\frac{2}{p}) \cdot 
    \\
    & \qquad \qquad \cdot \bigl(|\alpha_{i}-\alpha_{i-1}|^2_{L^2(\Omega)}+\kappa|\nabla (\alpha_{i}-\alpha_{i-1} )|^2_{[L^2(\Omega)]^2}\bigr),
  \nonumber
  \\
  &\hspace{20ex}\mbox{for any } i = 1, 2, \dots ,m.
  \label{inequ005}
\end{align}

Now, let us set
\begin{gather}\label{ken_tau_*}
\tau_*:=\frac{1}{4{C_*}(1 +|\nabla\gamma|_{W^{1, \infty}(\R^2; \R^2)})^2(1+c_{E_0}^\frac{2}{p})}.
\end{gather}
Then, having in mind
\begin{align}\label{contau}
 {2\tau (1 +|\nabla\gamma|_{W^{1, \infty}(\R^2; \R^2)})^2(1+c_{E_0}^{\frac{2}{p}})} \leq \frac{1}{2}, \mbox{ for } \tau \in ( 0 , \tau_* ),
\end{align}
we can verify the energy-dissipation \eqref{f-ene0} for $ \tau \in ( 0 , \tau_* ) $, by means of induction argument.

In fact, when $ i=1 $, it is observed from \eqref{contau} that
\begin{align*}
  &\frac{1}{\tau}\bigl(|\alpha_{1}-\alpha_{0}|^2_{L^2(\Omega)}+\kappa|\nabla(\alpha_{1}-\alpha_{0})|^2_{[L^2(\Omega)]^2}+|u_{1}-u_{0}|^2_{L^2(\Omega)}
  \\
  &\quad+\mu|\nabla(u_{1}-u_{0})|^2_{[L^2(\Omega)]^2}\bigr)+E(u_{1},\alpha_{1})-E(u_{0},\alpha_{0})
  \\
    &\leq 2 {C_*}(1 +|\nabla\gamma|_{W^{1, \infty}(\R^2; \R^2)})^2 (1+E(u_{0},\alpha_{0})^{\frac{2}{p}})(|\alpha_{1}-\alpha_{0}|^2_{L^2(\Omega)}+\kappa|\nabla(\alpha_{1}-\alpha_{0})|^2_{[L^2(\Omega)]^2}) 
  \\
  &\leq{2 \tau {C_*}(1 +|\nabla\gamma|_{W^{1, \infty}(\R^2; \R^2)})^2(1+c_{E_0}^{\frac{2}{p}})}\cdot\frac{1}{\tau}(|\alpha_{1}-\alpha_{0}|^2_{L^2(\Omega)}+\kappa|\nabla(\alpha_{1}-\alpha_{0})|^2_{[L^2(\Omega)]^2}) 
  \\
    &\leq\frac{1}{2\tau}\bigl(|\alpha_{1}-\alpha_{0}|^2_{L^2(\Omega)}+\kappa|\nabla(\alpha_{1}-\alpha_{0})|^2_{[L^2(\Omega)]^2} \bigr), \mbox{ for any $ \tau \in ( 0, \tau_* ) $.}
\end{align*}
Hence, \eqref{f-ene0} holds at the initial time-step $ i = 1 $. 

Subsequently, if \eqref{f-ene0} holds when $ i = k $ for all $ 1 \leq k \leq m -1$, then the assumption of induction immediately derives
\begin{gather*}
E ( u_{k} , \alpha_{k} ) \leq E ( u_{k-1} , \alpha_{k-1} ) \leq \dots \leq E ( u_{1} , \alpha_{1} ) \leq E ( u_{0} , \alpha_{0}) = c_{E_0}.
\end{gather*}
Therefore, invoking \eqref{inequ005}, we can verify \eqref{f-ene0} for the next time-step $ i = k +1 $, as follows:
\begin{align*}
  &\frac{1}{\tau}(|\alpha_{k+1}-\alpha_{k}|^2_{L^2(\Omega)}+\kappa|\nabla(\alpha_{k+1}-\alpha_{k})|^2_{[L^2(\Omega)]^2}+|u_{k+1}-u_{k}|^2_{L^2(\Omega)}
  \\
  &\quad+\mu|\nabla(u_{k+1}-u_{k})|^2_{[L^2(\Omega)]^2})+E(u_{k+1},\alpha_{k+1})-E(u_{k},\alpha_{k})
  \\
  &\leq 2{C_*}(1 +|\nabla\gamma|_{W^{1, \infty}(\R^2; \R^2)})^2(1+E(u_{k},\alpha_{k})^\frac{2}{p})(|\alpha_{k+1}-\alpha_{k}|^2_{L^2(\Omega)}+\kappa|\nabla(\alpha_{k+1}-\alpha_{k})|^2_{[L^2(\Omega)]^2})
  \\
  &\leq 2 \tau {C_*}(1 +|\nabla\gamma|_{W^{1, \infty}(\R^2; \R^2)})^2(1+c_{E_0}^\frac{2}{p})\frac{1}{\tau}(|\alpha_{k+1}-\alpha_{k}|^2_{L^2(\Omega)}+\kappa|\nabla(\alpha_{k+1}-\alpha_{k})|^2_{[L^2(\Omega)]^2})
  \\
  &\leq \frac{1}{2\tau}\bigl(|\alpha_{k+1}-\alpha_{k}|^2_{L^2(\Omega)}+\kappa|\nabla(\alpha_{k+1}-\alpha_{k})|^2_{[L^2(\Omega)]^2}\bigr), \mbox{ for any }\tau\in(0,\tau_*).
\end{align*}

Thus, we conclude Theorem \ref{003Thm1}. 
\end{proof}

\section{Proof of Main Theorem 1.}
    First, we define a regularization sequence $ \{ \gamma_\varepsilon \}_{\varepsilon \in (0, 1)} \subset C^{1, 1}(\R^2) \cap C^2(\R^2) $, by setting:
  \begin{gather*}
    \gamma_\varepsilon:=\left\{
      \begin{aligned}
          &\rho_\varepsilon*\gamma,&\mbox{ if }\gamma\notin C^{1,1}(\R^2) \cap C^2(\R^2),
        \\
          &\gamma,&\mbox{ if } \gamma\in C^{1,1}(\R^2) \cap C^2(\R^2),
      \end{aligned}\right.
      \\
      {\leftline{\mbox{with use of the standard mollifier $ \rho_\varepsilon \in C_\mathrm{c}^\infty(\R^2) $, for any $ \varepsilon \in (0, 1) $.}}}
  \end{gather*}
    Besides, we define an approximating free-energy $E_\varepsilon$ on $[L^2(\Omega)]^2$, as follows:
    \begin{gather}
    E_\varepsilon :[u,\alpha]\in W^{1,p}_0(\Omega) \times H^1_0(\Omega) \mapsto E_\varepsilon(u,\alpha):=\frac{1}{2}\int_\Omega|\nabla \alpha|^2\,dx
    +\frac{\nu}{p}\int_\Omega|\nabla u|^p\,dx\nonumber
    \\
    +\int_\Omega{\gamma}_\varepsilon(R(\alpha)\nabla u)\,dx
    +\frac{\lambda}{2}\int_\Omega|u-u_{\mathrm{org}}|^2\,dx\in[0,\infty).\label{E_eps}
  \end{gather}
  As is easily seen, 
    \begin{gather}\label{ken_gamma}
        \begin{cases}
            \gamma_\varepsilon \to \gamma \mbox{ uniformly on $ \R^2 $ as $ \varepsilon \downarrow 0 $,}
            \\
            |\nabla \gamma_\varepsilon|_{L^\infty(\R^2; \R^2)} \leq |\nabla \gamma|_{L^\infty(\R^2; \R^2)}, \mbox{ for all $ \varepsilon \in (0, 1) $.}
        \end{cases}
    \end{gather}
    The uniform convergence of $ \{ \gamma_\varepsilon \}_{\varepsilon \in (0, 1)} $ will imply
\begin{gather*}
    E_\varepsilon\rightarrow E\mbox{ on }[L^2(\Omega)]^2,\mbox{ in the sense of $\Gamma$-convergence, as }\varepsilon\downarrow0,
\end{gather*}
    and the uniform estimate for $ \{ |\nabla \gamma_\varepsilon|_{L^\infty(\R^2; \R^2)} \}_{\varepsilon \in (0, 1)} $ will lead to
   \begin{equation}\label{constc_E}
       c_E:=\sup_{\varepsilon\in(0,1)}E_\varepsilon(u_{0},\alpha_{0}) \leq E(\alpha_0, u_0) +|\nabla \gamma|_{L^\infty(\R^2; \R^2)}{|\Omega| <\infty,}
   \end{equation}
    {where $ |\Omega| $ denotes the Lebesgue measure of $ \Omega \subset \R^2 $. }

    Now, we take any $ \varepsilon \in (0, 1) $, and define an $ \varepsilon $-dependent constant $ \tau_{*}(\varepsilon) $ as follows:
\begin{gather*}
    \tau_{*}(\varepsilon):=\frac{1}{4C_*\bigl(1 +|\nabla\gamma_\varepsilon|_{W^{1, \infty}(\R^2; \R^2)} \bigr)^2(1+c_{E}^\frac{2}{p})}.
    \end{gather*}
    Invoking \eqref{ken_tau_*}, and applying Theorem \ref{003Thm1} to the case when $\gamma = \gamma_\varepsilon$, we will find a sequence of functional pairs $ \{ [u_{\varepsilon, i}, \alpha_{\varepsilon, i}] \}_{i = 1}^m $, for $ \varepsilon \in (0, 1) $, such that:
    \begin{gather}\label{ken_varIneq01}
    \frac{1}{\tau}(\alpha_{\varepsilon,i} - \alpha_{\varepsilon,i-1}, \varphi)_{L^2(\Omega)} +
    (\nabla \alpha_{\varepsilon,i}, \nabla \varphi)_{[L^2(\Omega)]^2} +
    \frac{\kappa}{\tau}(\nabla(\alpha_{\varepsilon,i} - \alpha_{\varepsilon,i-1}), \nabla \varphi)_{[L^2(\Omega)]^2} 
    \\
    + (\nabla \gamma_\varepsilon (R(\alpha_{\varepsilon,i-1}) \nabla u_{\varepsilon,i-1}) \cdot
    R \left(\alpha_{\varepsilon,i-1} + \ts\frac{\pi}{2} \right) \nabla u_{\varepsilon,i-1}, \varphi)_{L^2(\Omega)} = 0,
    \end{gather}
    \begin{gather}\label{ken_varIneq02}
    \frac{1}{\tau}(u_{\varepsilon,i} - u_{\varepsilon,i-1}, \psi)_{L^2(\Omega)} +
    (\nabla \gamma_\varepsilon (R(\alpha_{\varepsilon,i}) \nabla u_{\varepsilon,i}), R(\alpha_{\varepsilon,i}) \nabla \psi)_{[L^2(\Omega)]^2} 
    \\
    + \nu \int_\Omega |\nabla u_{\varepsilon,i}|^{p-2} \nabla u_{\varepsilon,i} \cdot \nabla \psi\,dx +
    \frac{\mu}{\tau}(\nabla(u_{\varepsilon,i} - u_{\varepsilon,i-1}), \nabla \psi)_{[L^2(\Omega)]^2} 
    \\
    + \lambda (u_{\varepsilon,i} - u_{\mathrm{org}}, \psi)_{L^2(\Omega)} = 0,
    \end{gather}
    and
  \begin{gather}
    \frac{1}{2\tau}
    |\alpha_{\varepsilon,i}-\alpha_{\varepsilon,i-1}|^2_{L^2(\Omega)}
    +\frac{\kappa}{2\tau}
    |\nabla(\alpha_{\varepsilon,i}-\alpha_{\varepsilon,i-1})|^2_{[L^2(\Omega)]^2}
    +\frac{1}{\tau}
    |u_{\varepsilon,i}-u_{\varepsilon,i-1}|^2_{L^2(\Omega)}
    \nonumber
    \\
    +\frac{\mu}{\tau}
    |\nabla(u_{\varepsilon,i}-u_{\varepsilon,i-1})|^2_{[L^2(\Omega)]^2}
    +E_\varepsilon(u_{\varepsilon,i},\alpha_{\varepsilon,i})
    \leq E_\varepsilon(u_{\varepsilon,i-1},\alpha_{\varepsilon,i-1}),\label{f-e-ineq}
    \\
      \mbox{ for all $ \varepsilon \in (0, 1) $, $ \tau \in (0, \tau_{*}(\varepsilon)) $, $ [\varphi, \psi] \in H_0^1(\Omega) \times W^{1, p}_0(\Omega) $, and $ i = 1, \dots, m $. }
    \nonumber
  \end{gather}
  {Additionally, taking into account Notation \ref{deftseq} and Example \ref{ex1}, we observe that}
  \begin{align}
    &\nabla\gamma_{\varepsilon}(R( [\underline{\alpha}_\varepsilon]_\tau(t))\nabla [\underline{u}_\varepsilon]_\tau(t))\in\partial\Phi_{\varepsilon}(R( [\underline{\alpha}_\varepsilon]_\tau(t))\nabla [\underline{u}_\varepsilon]_\tau(t))\mbox{ in }[L^2(\Omega)]^2,
    \label{subdig1}
    \\
    &\hspace{12ex}\mbox{ for any }\varepsilon\in(0,1), \mbox{ and a.e. }t\in(0,T),
    \nonumber
  \end{align}
  and hence, 
  \begin{gather}
    \nabla\gamma_{\varepsilon}(R( [\underline{\alpha}_\varepsilon]_\tau)\nabla [\underline{u}_\varepsilon]_\tau)\in\partial\widehat{\Phi}^I_{\varepsilon}(R( [\underline{\alpha}_\varepsilon]_\tau)\nabla [\underline{u}_\varepsilon]_\tau)\mbox{ in }L^2(0,T;[L^2(\Omega)]^2),
    \label{subdig2}
    \\
    \mbox{ for any }\varepsilon\in(0,1), \mbox{ and any open interval }I\subset(0,T),
    \nonumber
  \end{gather}
  where  $[{u}_\varepsilon]_\tau$, $ [\overline{u}_\varepsilon]_\tau $, and $ [\underline{u}_\varepsilon]_\tau $  are the time-interpolation of $ \{ u_{\varepsilon, i} \}_{i = 1}^m $ as in Notation \ref{deftseq}, and $[{\alpha}_\varepsilon]_\tau$, $ [\overline{\alpha}_\varepsilon]_\tau $, and $ [\underline{\alpha}_\varepsilon]_\tau $ are those of $ \{ \alpha_{\varepsilon, i} \}_{i = 1}^m $,  for $ \varepsilon \in (0, 1) $ and $ \tau \in (0, \tau_{*}(\varepsilon)) $.
\bigskip

From \eqref{ken_gamma}--\eqref{f-e-ineq}, we can see the following boundedness:
  \begin{description}
    \item[(B-1)]$\{[u_\varepsilon]_{\tau}~|~\tau\in(0,\tau_*(\varepsilon)),~\varepsilon\in(0,1)\}$ is bounded in $L^\infty(0,T;W^{1,p}_0(\Omega))$ and
    \\ 
    in $W^{1,2}(0,T;H^1_0(\Omega))$,
    \item[(B-2)]$\{[\overline{u}_\varepsilon]_{\tau}~|~\tau\in(0,\tau_*(\varepsilon)),~\varepsilon\in(0,1)\}$ and $\{[\underline{u}_\varepsilon]_{\tau}~|~\tau\in(0,\tau_*(\varepsilon)),~\varepsilon\in(0,1)\}$ are bounded in $L^\infty(0,T;W^{1,p}_0(\Omega))$,
    \item[(B-3)]$\{[\alpha_\varepsilon]_{\tau}~|~\tau\in(0,\tau_*(\varepsilon)),~\varepsilon\in(0,1)\}$ is bounded in $W^{1,2}(0,T;H^1_0(\Omega))$,
    \item[(B-4)]$\{[\overline{\alpha}_\varepsilon]_{\tau}~|~\tau\in(0,\tau_*(\varepsilon)),~\varepsilon\in(0,1)\}$ and $\{[\underline{\alpha}_\varepsilon]_{\tau}~|~\tau\in(0,\tau_*(\varepsilon)),~\varepsilon\in(0,1)\}$ are bounded in $L^\infty(0,T;H^1_0(\Omega))$,
    \item[(B-5)] $\{\nabla\gamma_{\varepsilon}(R( [\underline{\alpha}_\varepsilon]_\tau)\nabla [\underline{u}_\varepsilon]_\tau)~|~\tau\in(0,\tau_*(\varepsilon)),~\varepsilon\in(0,1)\}$ is bounded in $L^\infty(Q;\R^2)$,
    \item[(B-6)]The function of time $t\in[0,T]\mapsto E_\varepsilon([\overline{u}_\varepsilon]_{\tau }(t),[\overline{\alpha}_\varepsilon]_{\tau}(t))\in[0,\infty)$ and $t\in [0,T]\mapsto E_\varepsilon([\underline{u}_\varepsilon]_{\tau}(t),[\underline{\alpha}_\varepsilon]_{\tau}(t))\in[0,\infty)$ are nonincreasing for every $0<\varepsilon<1$ and $0<\tau<\tau_*(\varepsilon)$. Moreover, $\{E_\varepsilon({u}_{0},{\alpha}_{0})~|~\varepsilon\in(0,1)\}$ is bounded, and hence, the class $\{E_\varepsilon([\overline{u}_\varepsilon]_{\tau},[\overline{\alpha}_\varepsilon]_{\tau}),~E_\varepsilon([\underline{u}_\varepsilon]_{\tau},[\underline{\alpha}_\varepsilon]_{\tau})~|~\tau\in(0,\tau_*(\varepsilon)),~\varepsilon\in(0,1)\} $ is bounded in $BV(0,T)$.
  \end{description}
On account of (B-1)--(B-5) and 2D-embedding $ W^{1, p}(\Omega) \subset C(\overline{\Omega}) $  with $  p > 2 $, 
we can apply the compactness theory of Aubin's type \cite[Corollary 4]{MR0916688}, and find a sequence $\{\varepsilon_n\}_{n\in\N}\subset(0,1)$ with $ \{ \tau_n \}_{n\in\N} \subset (0, 1) $, and a pair of functions $[u,{\alpha}]\in[L^2(0,T;L^2(\Omega))]^2$ with $\bm{w}^*\in L^\infty(Q;\R^2)$ such that: 
\begin{gather*}
    \varepsilon_n \downarrow 0, ~ \tau_n := \frac{1}{2} \bigl( \tau_*(\varepsilon_n) \wedge \varepsilon_n \wedge 1 \bigr) \downarrow 0, \mbox{ as $ n \to \infty $,}
\end{gather*}
  \begin{align}
    & u_n : =  [u_{\varepsilon_n}]_{\tau_n}\rightarrow u\mbox{ in }C(\overline{Q}),\mbox{ weakly in } W^{ 1 , 2 }( 0,T ; H^1_0 ( \Omega)),
    \nonumber 
    \\
      &\hspace*{10ex}\mbox{ weakly-}* \mbox{ in } L^\infty( 0,T ; W^{ 1 , p }_0 ( \Omega ) ) , \mbox{ as $ n \to \infty $,} 
    \label{conv_1}
    \\[1ex]
    &\alpha_n:=[\alpha_{\varepsilon_n}]_{\tau_n}\rightarrow\alpha\mbox{ in }C([0,T];L^2(\Omega)),\mbox{ weakly in }W^{1,2}(0,T;H^1_0(\Omega)),
    \nonumber
    \\
      &\hspace*{10ex}\mbox{ weakly-}*\mbox{in }L^\infty(0,T;H^1_0(\Omega)), \mbox{ as $ n \to \infty $,}
    \label{conv_2}
  \end{align}
  and
  \begin{align}
      &\nabla\gamma_{\varepsilon_n}(R( [\underline{\alpha}_{\varepsilon_n}]_{\tau_n})\nabla [\underline{u}_{\varepsilon_n}]_{\tau_n})~\rightarrow\bm{w}^*\mbox{ weakly-}*\mbox{in }L^\infty(Q;\R^2), \mbox{ as $ n \to \infty $,}
    \label{conv_40}
  \end{align}
  in particular,
  \begin{align}\label{conv_100}
    [u(0),\alpha(0)]=\lim_{n\rightarrow\infty}[u_n(0),\alpha_n(0)]=[u_0,\alpha_0]\mbox{ in }[L^2(\Omega)]^2.
  \end{align}
  Here, since
  \begin{align*}
    &\left\{
    \begin{aligned}
      &\max\{|([\overline{u}_{\varepsilon_n}]_{\tau_n}-u_{n})(t)|_{H^1(\Omega)},|([\underline{u}_{\varepsilon_n}]_{\tau_n}-u_{n})(t)|_{H^1(\Omega)}\}
      \\
      &\quad\leq \int_{t_{i-1}}^{t_i}|\partial_tu_{n}(t)|_{H^1(\Omega)}\,dt\leq \tau_n^{\frac{1}{2}}|\partial_tu_{n}|_{L^2(0,T;H^1(\Omega))},
      \\
      &\max\{|([\overline{\alpha}_{\varepsilon_n}]_{\tau_n}-\alpha_{n})(t)|_{H^1(\Omega)},|([\underline{\alpha}_{\varepsilon_n}]_{\tau_n}-\alpha_{n})(t)|_{H^1(\Omega)}\}
      \\
      &\quad\leq \int_{t_{i-1}}^{t_i}|\partial_t\alpha_{n}(t)|_{H^1(\Omega)}\,dt\leq \tau_n^{\frac{1}{2}}|\partial_t\alpha_{n}|_{L^2(0,T;H^1(\Omega))}, 
    \end{aligned}
    \right.
    \\
    &\qquad \mbox{ for }t\in[t_{i-1},t_i),~i=1,\dots,m, \mbox{ and }\tau\in(0,\tau_*(\varepsilon)),
  \end{align*}
  one can also see from \eqref{conv_1} and \eqref{conv_2} that:
  \begin{align}
    &\overline{u}_{n}:=[\overline{u}_{\varepsilon_n}]_{\tau_n}\rightarrow u,~\underline{u}_{n}:=[\underline{u}_{\varepsilon_n}]_{\tau_n}\rightarrow u \mbox{ in }L^\infty(0,T;L^2(\Omega)),
    \nonumber
    \\
    &\hspace*{10ex}\mbox{ weakly-}* \mbox{ in }L^\infty(0,T;W^{1,p}_0(\Omega)),
    \label{conv_3}
    \\[1ex]
    &\overline{\alpha}_{n}:=[\overline{\alpha}_{\varepsilon_n}]_{\tau_n}\rightarrow\alpha,~\underline{\alpha}_{n}:=[\underline{\alpha}_{\varepsilon_n}]_{\tau_n}\rightarrow\alpha\mbox{ in }L^\infty ( 0,T ; L^2 ( \Omega ) ),
    \nonumber
    \\
    &\hspace*{10ex}\mbox{ weakly-}*\mbox{ in }L^\infty(0,T;H^1_0(\Omega)),
    \label{conv_4}
  \end{align}
and in particular,
\begin{align}\label{conv_7}
  &\left\{
    \begin{aligned}
      &\overline{u}_{n}(t)\rightarrow u(t), \underline{u}_{n}(t)\rightarrow u(t)\mbox{ in }L^p(\Omega) \mbox{ and weakly in }W^{1,p}_0(\Omega),
      \\
      &\overline{\alpha}_{n}(t)\rightarrow \alpha(t), \underline{\alpha}_{n}\rightarrow \alpha(t)\mbox{ in }L^2(\Omega) \mbox{ and weakly in }H^1_0(\Omega),
    \end{aligned}\right.
    \\
    &\hspace{15ex}\mbox{ as } n\rightarrow\infty,\mbox{ for any }t\in(0,T).
    \nonumber
\end{align}
Moreover, \eqref{constc_E} and (B-6) enable us to see
\begin{gather}\label{conv_7-0}
    \bigl| E_{\varepsilon_n}(\overline{u}_n, \overline{\alpha}_n) -E_{\varepsilon_n}(\underline{u}_n, \underline{\alpha}_n) \bigr|_{L^1(0, T)} \leq 2c_E \tau_n \to 0, \mbox{ as $ n \to \infty $}.
\end{gather}
So, applying Helly's selection theorem \cite[Chapter 7, p.167]{rudin1976principles}, we will find a bounded and nonincreasing function $\mathcal{J}_*:[0,T]\mapsto[0,\infty)$, such that 
\begin{align}
    &E_{\varepsilon_n}(\overline{u}_{n},\overline{\alpha}_{n})\rightarrow\mathcal{J}_* \mbox{ and } E_{\varepsilon_n}(\underline{u}_{n},\underline{\alpha}_{n})\rightarrow\mathcal{J}_* \nonumber
  \\
  & \qquad \mbox{ weakly-}*\mbox{ in }BV(0,T),\mbox{ and }\mbox{weakly-}* \mbox{ in }L^\infty(0,T),
  \\
    &E_{\varepsilon_n} (\overline{u}_{n}(t),\overline{\alpha}_{n}(t)) \rightarrow \mathcal{J}_*(t) \mbox{ and } E_{\varepsilon_n} (\underline{u}_{n}(t),\underline{\alpha}_{n}(t)) \rightarrow \mathcal{J}_*(t), \mbox{ for any }t\in[0,T], \label{conv_7-1}
\end{align}
as $ n \to \infty $, by taking a subsequence if necessary.

Now, let us show the pair of function $[u,\alpha]$ is a solution to the system (S). (S4) can be checked by \eqref{conv_100}. Next, let us show $[u,\alpha]$ satisfies the variational inequalities (S1) and (S2). Let us take any $ t \in [0, T] $. Then, from \eqref{ken_varIneq01} and \eqref{ken_varIneq02}, the sequences as in \eqref{conv_1}--\eqref{conv_4} satisfy the following inequalities:
\begin{align}
  &\int_0^t (\partial_t \alpha_{n} (\sigma),\omega(\sigma))_{L^2(\Omega)}\,d\sigma +\int_{0}^{t}(\nabla\overline{\alpha}_{n}(\sigma),\nabla\omega(\sigma))_{[L^2(\Omega)]^2} \,d\sigma
  \nonumber
  \\
  &\quad + \int_0^t\int_\Omega\nabla\gamma_{\varepsilon_n}(R(\underline{\alpha}_{n}(\sigma))\nabla\underline{u}_{n}(\sigma))\cdot R\left(\underline{\alpha}_{n}(\sigma)+{\ts\frac{\pi}{2}}\right)\nabla \underline{u}_{n}(\sigma)\omega(\sigma)\,dxd\sigma
  \nonumber 
  \\ 
  &\quad+\kappa\int_0^t (\nabla\partial_t \alpha_{n}(\sigma),\nabla\omega(\sigma))_{[L^2(\Omega)]^2 } \, d\sigma= 0,\mbox{ for any }\omega\in L^2(0,T;H^1_0(\Omega)),
  \label{conv_8}
\end{align}
and, 
  \begin{align}
    &\int_0^t (\partial_tu_n(\sigma),\overline{u}_n(\sigma)-w(\sigma))_{L^2(\Omega)}\,d\sigma+\mu\int_0^t(\nabla\partial_tu_n(\sigma),\nabla(\overline{u}_n-w)(\sigma))_{[L^2(\Omega)]^2}\,d\sigma
    \nonumber
    \\
    &\quad+{\nu}\int_0^t\int_{\Omega}|\nabla \overline{u}_n(\sigma)|^{p-2}\nabla \overline{u}_n(\sigma)\cdot\nabla(\overline{u}_{n}-w)(\sigma)\,dxd\sigma
    \nonumber
    \\
    &\quad+\lambda\int_0^t(\overline{u}_{n}(\sigma)-u_{\mathrm{org}},(\overline{u}_{n}-w)(\sigma))_{L^2(\Omega)}\,d\sigma
    \nonumber
    \\
    &\quad+\int_0^t\int_\Omega\gamma_{\varepsilon_n}(R(\overline{\alpha}_{n}(\sigma))\nabla\overline{u}_{n} (\sigma))\,dxd\sigma\leq\int_0^t\int_\Omega\gamma_{\varepsilon_n}(R(\overline{\alpha}_{n}(\sigma))\nabla w(\sigma))\,dxd\sigma,
    \nonumber
    \\
    &\hspace{25ex}\mbox{ for any }w\in L^2(0,T;W^{1,p}_0(\Omega)).\label{conv_9}
  \end{align}
  Here, from \eqref{conv_3} and \eqref{conv_4} it is seen that
  \begin{align}
    &\liminf_{n\rightarrow\infty}\int_0^t|\nabla\overline{u}_n(\sigma)|_{[L^p(\Omega)]^2}^p\,d\sigma\geq\int_0^t|\nabla{u}(\sigma)|_{[L^p(\Omega)]^2}^p\,d\sigma,
    \label{conv_10}
    \\
    &\liminf_{n\rightarrow\infty}\int_0^t|\nabla\overline{\alpha}_n(\sigma)|_{[L^2(\Omega)]^2}^2\,d\sigma\geq\int_0^t|\nabla{\alpha}(\sigma)|_{[L^2(\Omega)]^2}^2\,d\sigma,
    \label{conv_23}
  \end{align}
  Also, using \eqref{ken_gamma}, \eqref{conv_3} and \eqref{conv_4}, weakly lower-semicontinuity of $\gamma_\varepsilon$ and Fatou's lemma, one can see  
  \begin{align}
      & \liminf_{n\rightarrow\infty}\int_{0}^t\int_\Omega\gamma_{\varepsilon_n}(R(\overline{\alpha}_{n}(\sigma))\nabla\overline{u}_{n}(\sigma))\,dxd\sigma 
      \nonumber
      \\
      & \geq 
    \liminf_{n\rightarrow\infty}\int_{0}^{t}\int_\Omega\gamma(R(\overline{\alpha}_n(\sigma)\nabla \overline{u}_n(\sigma)))\,dxd\sigma
      \nonumber
      \\
      & \qquad +\lim_{n\rightarrow\infty}|\gamma_{\varepsilon_n}-\gamma|_{L^\infty(\R^2)}\int_{0}^{t}\int_\Omega1\,dxd\sigma
    \nonumber
    \\
    & \geq \int_0^t\int_\Omega\gamma(R({\alpha}(\sigma))\nabla{u}(\sigma))\,dxd\sigma.
    \label{conv_48} 
  \end{align}
  {Furthermore, putting $m_t:=\bigl([\frac{t}{\tau}]+1\bigr)\land \frac{T}{\tau}$ {with use of the integral part $ [\frac{t}{\tau}] \in \mathbb{Z} $ of $ \frac{t}{\tau} \in \R $,} we observe that}
  \begin{align*}
    &\int_0^t(\nabla\partial_tu_n(\sigma),\nabla\overline{u}_n(\sigma))_{[L^2(\Omega)]^2}\,d\sigma
    \\
    &=\int_{0}^{t_{{m}_t}}(\nabla\partial_tu_n(\sigma),\nabla\overline{u}_n(\sigma))_{[L^2(\Omega)]^2}\,d\sigma-\int_{t}^{t_{m_t}}(\nabla\partial_tu_n(\sigma),\nabla\overline{u}_n(\sigma))_{[L^2(\Omega)]^2}\,d\sigma
    \\
    &\geq\sum^{m_t}_{i=1}\frac{1}{2}(|\nabla u_n(t_i)|_{[L^2(\Omega)]^2}^2-|\nabla u_n(t_{i-1})|_{ [L^2(\Omega)]^2}^2)
    \\
    &\qquad-|\Omega|^{\frac{p-2}{2p}}|\nabla\overline{u}_n|_{L^\infty(0,T;[L^p(\Omega)]^2)}\int_{t}^{t_{m_t}}|\nabla\partial_tu_n(\sigma)|_{[L^2(\Omega)]^2}\,d\sigma
    \\
    &\geq\frac{1}{2}(|\nabla u_n(t_{m_t})|_{[L^2(\Omega)]^2}^2-|\nabla u_n(0)|_{ [L^2(\Omega)]^2}^2)
    \\
    &\quad-\tau_n^\frac{1}{2}|\Omega|^{\frac{p-2}{2p}}|\nabla\overline{u}_n|_{L^\infty(0,T;[L^p(\Omega)]^2)}|\nabla\partial_tu_n|_{L^2(0,T;[L^2(\Omega)]^2)}
    \\
    &=\frac{1}{2}(|\nabla\overline{u}_n(t)|^2_{[L^2(\Omega)]^2}-|\nabla u_0|_{[L^2(\Omega)]^2}^2)
      \\
      & \quad -\tau_n^\frac{1}{2}|\Omega|^{\frac{p-2}{2p}}|\nabla\overline{u}_n|_{L^\infty(0,T;[L^p(\Omega)]^2)}|\nabla\partial_tu_n|_{L^2(0,T;[L^2(\Omega)]^2)},
  \end{align*}
  and similarly,
  \begin{align*}
    &\int_0^t(\nabla\partial_t\alpha_n(\sigma),\nabla\overline{\alpha}_n(\sigma))_{[L^2(\Omega)]^2}\,d\sigma
    \\
    &=\frac{1}{2}(|\nabla\overline{\alpha}_n(t)|^2_{[L^2(\Omega)]^2}-|\nabla \alpha_0|_{[L^2(\Omega)]^2}^2)-\tau_n^\frac{1}{2}|\nabla\overline{\alpha}_n|_{L^\infty(0,T;[L^2(\Omega)]^2)}|\nabla\partial_t\alpha_n|_{L^2(0,T;[L^2(\Omega)]^2)},
      \\
      &\hspace{25ex} \mbox{for $ n = 1, 2, 3, \dots $.}
  \end{align*}
  Hence, by \eqref{conv_7}, we obtain
  \begin{align}
    &\liminf_{n\rightarrow\infty}\int_0^t(\nabla\partial_tu_n(\sigma),\nabla\overline{u}_n(\sigma))_{[L^2(\Omega)]^2}\,d\sigma
    \label{conv_12}
    \\
    &\quad\geq \frac{1}{2}(|\nabla{u}(t)|^2_{[L^2(\Omega)]^2}-|\nabla u_0|_{[L^2(\Omega)]^2}^2)=\int_{0}^{t}(\nabla\partial_tu(\sigma),\nabla u(\sigma))_{[L^2(\Omega)]^2}\,d\sigma,
    \nonumber
    \\
    &\liminf_{n\rightarrow\infty}\int_0^t(\nabla\partial_t\alpha_n(\sigma),\nabla\overline{\alpha}_n(\sigma))_{[L^2(\Omega)]^2}\,d\sigma
    \label{conv_13}
    \\
    &\quad\geq\frac{1}{2}(|\nabla{\alpha}(t)|^2_{[L^2(\Omega)]^2}-|\nabla \alpha_0|_{[L^2(\Omega)]^2}^2)=\int_{0}^{t}(\nabla\partial_t\alpha(\sigma),\nabla \alpha(\sigma))_{[L^2(\Omega)]^2}\,d\sigma.
    \nonumber
  \end{align}
  {On the other hand, by putting $w=u$ in \eqref{conv_9} and using \eqref{conv_1}, \eqref{conv_3} and \eqref{conv_4}, we see that}
   \begin{align}
    &\limsup_{n\rightarrow\infty} 
    \biggl( 
      \frac{\nu}{p}\int_0^t|\nabla\overline{u}_n(\sigma)|_{[L^p(\Omega)]^2}^p\,d\sigma+\int_0^t\int_\Omega\gamma_{\varepsilon_n}(R( \overline{\alpha}_n(\sigma))\nabla\overline{u}_n(\sigma))\,dxd\sigma
      \nonumber
      \\
     & \hspace{1ex}
        +\frac{\mu}{2}(|\nabla\overline{u}_n(t)|_{[L^2(\Omega)]^2}^2-|\nabla u_0|_{[L^2(\Omega)]^2}^2)-\tau_n^\frac{1}{2}|\Omega|^\frac{p-2}{2p}|\nabla\overline{u}_n|_{L^\infty(0,T;[L^p(\Omega)]^2)}|\nabla\partial_tu_n|_{L^2(0,T;[L^2(\Omega)]^2)}
    \biggr)
    \nonumber
    \\
    &\leq \limsup_{n\rightarrow\infty} 
    \biggl( 
      \frac{\nu}{p}\int_0^t|\nabla\overline{u}_n(\sigma)|_{[L^p(\Omega)]^2}^p\,d\sigma+\int_0^t\int_\Omega\gamma_{\varepsilon_n}(R( \overline{\alpha}_n(\sigma))\nabla\overline{u}_n(\sigma))\,dxd\sigma
      \nonumber
      \\
     & \hspace{15ex}
        +\mu\int_0^t(\nabla\partial_tu_n(\sigma),\nabla\overline{u}_n(\sigma))_{[L^2(\Omega)]^2}\,d\sigma
    \biggr)\label{conv_22}
    \\
    & \leq - \lim_{n\rightarrow\infty}\int_0^t (\partial_tu_n(\sigma)+\lambda(\overline{u}_n(\sigma)-u_{\mathrm{org}}),(\overline{u}_n-u)(\sigma))_{L^2(\Omega)}\,d\sigma 
    \nonumber
    \\
    &\quad+\lim_{n\rightarrow\infty}\mu\int_0^t(\nabla\partial_tu_n(\sigma),\nabla u(\sigma))_{[L^2(\Omega)]^2}\,d\sigma+\frac{\nu}{p}\int_0^t|\nabla{u}(\sigma)|_{[L^p( \Omega )]^2}^p\,d\sigma
    \nonumber
    \\
    &\quad+\lim_{n\rightarrow\infty}\int_0^t\int_\Omega\gamma_{\varepsilon_n}(R(\overline{\alpha}_n(\sigma))\nabla u(\sigma))\,dxd\sigma
    \nonumber
    \\
    &=\frac{\nu}{p}\int_0^t|\nabla{u}(\sigma)|_{[L^p( \Omega )]^2}^p\,d\sigma+\int_{0}^t\int_\Omega\gamma(R({\alpha}(\sigma))\nabla{u}(\sigma))\,dxd\sigma
    \nonumber
    \\
    &\quad + \mu\int_0^t(\nabla\partial_tu(\sigma),\nabla u(\sigma))_{[L^2(\Omega)]^2}\,d\sigma
    \nonumber
    \\
    &=\frac{\nu}{p}\int_0^t|\nabla{u}(\sigma)|_{[L^p( \Omega )]^2}^p\,d\sigma+\int_{0}^t\int_\Omega\gamma(R({\alpha}(\sigma))\nabla{u}(\sigma))\,dxd\sigma
    \nonumber
    \\
    &\quad +\frac{\mu}{2}(|\nabla{u}(t)|^2_{[L^2(\Omega)]^2}-|\nabla u_0|_{[L^2(\Omega)]^2}^2).
    \nonumber
   \end{align}
   Therefore, from (Fact 1) in Section 1, \eqref{conv_10}, \eqref{conv_48}, \eqref{conv_12} and \eqref{conv_22}, we can derive the following convergences as $n\rightarrow\infty$:
   \begin{align}
      &\int_0^t|\nabla\overline{u}_n(\sigma)|_{[L^p(\Omega)]^2}^p\,d\sigma\rightarrow\int_0^t|\nabla{u}(\sigma)|_{[L^p(\Omega)]^2}^p\,d\sigma,
      \label{conv_14}
    \\
    &\int_0^t(\nabla\partial_tu_n(\sigma),\nabla\overline{u}_n(\sigma))_{[L^2(\Omega)]^2}\,d\sigma\rightarrow\int_0^t(\nabla\partial_tu(\sigma),\nabla{u}(\sigma))_{[L^2(\Omega)]^2}\,d\sigma,
    \label{conv_16}
    \\
       &|\nabla\overline{u}_n(t)|_{[L^2(\Omega)]^2}^2\rightarrow|\nabla{u}(t)|_{[L^2(\Omega)]^2}^2, \mbox{ for any $ t \in [0, T] $.}
      \label{conv_17}
   \end{align} 
  Moreover, owing to \eqref{ken_gamma}, \eqref{constc_E}, (B-6), and uniform convexity of $ L^2 $- and $ L^p $-based topologies, we can infer that:
  \begin{gather}\label{conv_7-2}
      \hspace{-6ex}\begin{cases}
          \hspace{-2ex}
          \parbox{12cm}{
              \vspace{-2ex}
              \begin{itemize}
                  \item $ \overline{u}_n \to u \mbox{ and } \underline{u}_n \to u $  in $ L^p(0, T; W_0^{1, p}(\Omega)) $,
                      \\
                      with $ \bigl| |\nabla \overline{u}_n|_{L^p(0, T; [L^p(\Omega)]^2)}^p -|\nabla \underline{u}_n|_{L^p(0, T; [L^p(\Omega)]^2)}^p \bigr| \leq \frac{2 p c_E}{\nu} \tau_n \to 0 $,
                  \item $ \overline{u}_n(t) \to u(t) \mbox{ and } \underline{u}_n(t) \to u(t) $  in $ W_0^{1, p}(\Omega) $, a.e. $ t \in (0, T) $,
              \end{itemize}
              \vspace{-2ex}
          }
      \end{cases} \mbox{as $ n \to \infty $,}
  \end{gather}
  and
  \begin{gather}\label{conv_7-3}
      \begin{cases}
          \hspace{-2ex}
          \parbox{14cm}{
              \vspace{-2ex}
              \begin{itemize}
                  \item $ \overline{u}_n(t) \to u(t) $ and $ \underline{u}_n(t) \to u(t) $  in $ H_0^{1}(\Omega) $, with
                      \\
                      $ \bigl| \nabla (\overline{u}_n -\underline{u}_n)(t)\bigr|_{[L^2(\Omega)]^2} \leq {\tau_n}^{\frac{1}{2}} |\nabla \partial_t u_n|_{L^2(0, T; [L^2(\Omega)]^2)} \to 0 $,
                  \item $ \bigl| \gamma_{\varepsilon_n}(R(\overline{\alpha}_n(t)) \nabla \overline{u}_n(t)) -\gamma(R(\alpha(t)) \nabla u(t)) \bigr|_{L^1(\Omega)} $
                      \\
                      $ \leq |\gamma_{\varepsilon_n}-\gamma|_{L^\infty(\R^2)}|\nabla\overline{u}_n(t)|_{[L^2(\Omega)]^2}|\Omega|^{\frac{1}{2}}+4|\nabla \gamma|_{L^\infty(\R^2; \R^2)}|\nabla\overline{u}_n(t)|_{[L^2(\Omega)]^2}\cdot
                      \\
                      \cdot|(\alpha_n-\alpha)(t)|_{L^2(\Omega)}+|\nabla \gamma|_{L^\infty(\R^2; \R^2)}|\Omega|^{\frac{1}{2}}|\nabla (\overline{u}_n -u)(t)|_{[L^2(\Omega)]^2} \to 0$,\\[1ex]
                      as well as $ \bigl| \gamma_{\varepsilon_n}(R(\underline{\alpha}_n(t)) \nabla \underline{u}_n(t)) -\gamma(R(\alpha(t)) \nabla u(t)) \bigr|_{L^1(\Omega)} \to 0$, 
              \end{itemize}
              \vspace{-2ex}
          }
      \end{cases}
      \\
      \mbox{for any $ t \in [0, T] $, as $ n \to \infty $.}
  \end{gather}
 \eqref{energy01}, \eqref{E_eps}, \eqref{conv_7}, \eqref{conv_7-0}, \eqref{conv_7-1}, \eqref{conv_7-2}, and \eqref{conv_7-3} imply
 \begin{gather}\label{convJ}
     \begin{cases}
         E_{\varepsilon_n}(\overline{u}_n(t), \overline{\alpha}_n(t)) \to \mathcal{J}_*(t) = E(u(t), \alpha(t)),
         \\
         E_{\varepsilon_n}(\underline{u}_n(t), \underline{\alpha}_n(t)) \to \mathcal{J}_*(t) = E(u(t), \alpha(t)), 
     \end{cases}
     \mbox{a.e. $ t \in (0, T) $, as $ n \to \infty $.}
 \end{gather}

  On the other hand, if we take $\omega=\overline{\alpha}_n-\alpha$ in \eqref{conv_8}, then having in mind \eqref{conv_2}, \eqref{conv_4}, \eqref{conv_13} and \eqref{conv_14}, we see that
   \begin{align}
    &\limsup_{n\rightarrow\infty}
    \biggl( 
      \frac{1}{2}\int_0^t|\nabla\overline{\alpha}_n(\sigma)|_{[L^2(\Omega)]^2}^2\,d\sigma+\frac{\kappa}{2}(|\nabla\overline{\alpha}_n(t)|_{[L^2(\Omega)]^2}^2-|\nabla \alpha_0|_{[L^2(\Omega)]^2}^2)
      \nonumber
      \\
     & \hspace{15ex}
        -\tau_n^\frac{1}{2}|\nabla\overline{\alpha}_n|_{L^\infty(0,T;[L^2(\Omega)]^2)}|\nabla\partial_t\alpha_n|_{L^2(0,T;[L^2(\Omega)]^2)}
    \biggr)
    \nonumber
    \\
    &\leq \limsup_{n\rightarrow\infty} 
    \biggl( 
      \frac{1}{2}\int_0^t|\nabla\overline{\alpha}_n(\sigma)|_{[L^2(\Omega)]^2}^2\,d\sigma+\kappa\int_0^t (\nabla\partial_t \alpha_n(\sigma),\nabla\overline{\alpha}_n(\sigma))_{[L^2(\Omega)]^2 } \, d\sigma
    \biggr)
    \label{conv_20}
    \\
    &\leq -\lim_{n\rightarrow\infty}\int_0^t\int_\Omega\nabla\gamma_{\varepsilon_n}(R(\underline{\alpha}_n(\sigma))\nabla\underline{u}_n(\sigma))\cdot R\left(\underline{\alpha}_n(\sigma)+\ts\frac{\pi}{2}\right)\nabla \underline{u}_n(\sigma)(\overline{\alpha}_n-\alpha)(\sigma)\,dxd\sigma
    \nonumber
    \\
    &\quad - \lim_{n\rightarrow\infty}\int_0^t (\partial_t \alpha_n (\sigma),(\overline{\alpha}_n-\alpha)(\sigma))_{L^2(\Omega)}\,d\sigma+\frac{1}{2}\int_{0}^{t}|\nabla\alpha(\sigma)|^2_{[L^2(\Omega)]^2}\,d\sigma
    \nonumber
    \\
    &\quad + \lim_{n\rightarrow\infty}\kappa\int_0^t (\nabla\partial_t \alpha_n(\sigma),\nabla{\alpha}(\sigma))_{[L^2(\Omega)]^2 } \, d\sigma
    \nonumber
    \\
    &=\frac{1}{2}\int_{0}^{t}|\nabla\alpha(\sigma)|^2_{[L^2(\Omega)]^2}\,d\sigma+\kappa\int_0^t (\nabla\partial_t \alpha(\sigma),\nabla{\alpha}(\sigma))_{[L^2(\Omega)]^2 } \, d\sigma
    \nonumber
    \\
    &=\frac{1}{2}\int_{0}^{t}|\nabla\alpha(\sigma)|^2_{[L^2(\Omega)]^2}\,d\sigma+\frac{\kappa}{2}(|\nabla{\alpha}(t)|^2_{[L^2(\Omega)]^2}-|\nabla \alpha_0|_{[L^2(\Omega)]^2}^2).
    \nonumber
   \end{align}
   Therefore, from (Fact 1) in {Section 1}, \eqref{conv_23}, \eqref{conv_13} and \eqref{conv_20}, we can derive the following convergences as $n\rightarrow\infty$:
   \begin{align*}
    &\int_0^t|\nabla\overline{\alpha}_n(\sigma)|_{[L^2(\Omega)]^2}^2\,d\sigma\rightarrow\int_0^t|\nabla{\alpha}(\sigma)|_{[L^2(\Omega)]^2}^2\,d\sigma,
  \\
  &\int_0^t(\nabla\partial_t\alpha_n(\sigma),\nabla\overline{\alpha}_n(\sigma))_{[L^2(\Omega)]^2}\,d\sigma\rightarrow\int_0^t(\nabla\partial_t\alpha(\sigma),\nabla{\alpha}(\sigma))_{[L^2(\Omega)]^2}\,d\sigma,
  \\
    &|\nabla\overline{\alpha}_n(t)|_{[L^2(\Omega)]^2}^2\rightarrow|\nabla{\alpha}(t)|_{[L^2(\Omega)]^2}^2,\mbox{ and therefore, }\overline{\alpha}_n(t)\rightarrow \alpha(t)\mbox{ in }H^1_0(\Omega),
 \end{align*}
 for any $t\in[0,T]$. 

 {Now, set $w=\psi$ in $W^{1,p}_0(\Omega)$ in \eqref{conv_9} and $\omega=\varphi$ in $H^{1}_0(\Omega)$ in \eqref{conv_8}. In view of \eqref{conv_1}, \eqref{conv_3}, \eqref{conv_4}, \eqref{conv_14} and \eqref{conv_16}, letting $ n \rightarrow \infty $ yields that}
    \begin{align}
      &\int_I(\partial_tu(t),u(t)-\psi)_{L^2(\Omega)}\,dt+\lambda\int_I(u(t)-u_\mathrm{org},u(t)-\psi)_{L^2(\Omega)}\,dt
      \nonumber
      \\
      &\quad+\mu\int_I(\nabla\partial_tu(t),\nabla(u(t)-\psi))_{[L^2(\Omega)]^2}\,dt
      \nonumber
      \\
      &\quad+\nu\int_I\int_{\Omega}|\nabla u(t)|^{p-2}\nabla u(t)\cdot\nabla(u(t)-\psi)\,dxdt
      \nonumber
      \\
      &\quad+\int_I\int_{\Omega}\gamma(R(\alpha(t))\nabla u(t))\,dxdt\leq\int_I\int_{\Omega}\gamma (R(\alpha(t))\nabla\psi)\,dxdt,
      \label{conv_28}
    \end{align}
    {for any open interval $I\subset(0,T)$. Also, using} \eqref{conv_2}, \eqref{conv_4}, and \eqref{conv_14}, we let $ n \rightarrow \infty $ and obtain that 
    \begin{align}
      &\int_I (\partial_t \alpha (t),\varphi)_{L^2(\Omega)}\,dt +\int_I(\nabla{\alpha}(t),\nabla\varphi)_{[L^2(\Omega)]^2} \,dt+\kappa\int_I (\nabla\partial_t \alpha(t),\nabla\varphi)_{[L^2(\Omega)]^2 } \, dt
      \nonumber
      \\
      &\quad + \int_I\int_\Omega\bm{w}^*(t)\cdot R\left({\alpha}(t)+\ts\frac{\pi}{2}\right)\nabla{u}(t)\varphi\,dxdt= 0, 
      \label{conv_29}
    \end{align}\noeqref{conv_29}{for any open interval $I\subset(0,T)$.} Furthermore, invoking \eqref{subdig1}, \eqref{subdig2}, \eqref{conv_40}, Example \ref{ex2}, and (Fact 3) in Remark \ref{rem2}, one can say that 
    \begin{equation*}
     \bm{w}^*\in\partial\widehat{\Phi}_0^I(R(\alpha)\nabla u)\mbox{ in }L^2(0,T;[L^2(\Omega)]^2),
    \end{equation*}
 and hence, 
 \begin{equation}\label{subdig3}
  \begin{aligned}
    &\bm{w}^*\in\partial\Phi_0(R(\alpha)\nabla u) \mbox{ in } [L^2(\Omega)]^2, \mbox{ for a.e. }t\in(0,T),\mbox{ and }
    \\
    &\bm{w}^*\in\partial\gamma(R(\alpha)\nabla u) \mbox{ in }\R^2, \mbox{ a.e. in } Q.
  \end{aligned}
 \end{equation}
\eqref{conv_28}--\eqref{subdig3} imply that the pair $[u,\alpha]$ satisfies (S1) and (S2). 

    Next, we consider the energy inequality. From \eqref{f-e-ineq}, it is derived that 
    \begin{align}
      &\int_{t_{i-1}}^{t_i}\biggl(\frac{1}{2}|\partial_t\alpha_n(\sigma)|^2_{L^2(\Omega)}+\frac{\kappa}{2}|\nabla\partial_t\alpha_n(\sigma)|^2_{[L^2(\Omega)]^2}+|\partial_tu_n(\sigma)|^2_{L^2(\Omega)}+{\mu}|\nabla\partial_tu_n(\sigma)|^2_{[L^2(\Omega)]^2}\biggr)\,d\sigma
      \nonumber
      \\
      &\quad+E_{\varepsilon_n}(\overline{u}_n(t),\overline{\alpha}_n(t))\leq E_{\varepsilon_n}(\underline{u}_n(t),\underline{\alpha}_n(t)),\mbox{ for }t\in[t_{i-1},t_i),~i=1,2,\dots,{\ts\frac{T}{\tau_n}},\mbox{ and }n\in\N.
      \label{energy1}
    \end{align}
    {Here, putting $m^s:=[\frac{s}{\tau}]$ and $m_t:=\bigl([\frac{t}{\tau}]+1\bigr)\land \frac{T}{\tau}$ for $0\leq s< t\leq T$, and summing up the both side of \eqref{energy1} for $i=m^s+1, m^s+2,\dots,m_t$, we obtain that}
    \begin{align}
      &\frac{1}{2}\int_s^t \left( |\partial_t\alpha_n(\sigma)|^2_{L^2(\Omega)}+{\kappa}|\nabla\partial_t\alpha_n(\sigma)|^2_{[L^2(\Omega)]^2}\right)\,d\sigma
      \nonumber
      \\
      &\quad+\int_s^t \left(|\partial_tu_n(\sigma)|^2_{L^2(\Omega)}+{\mu}|\nabla\partial_tu_n(\sigma)|^2_{[L^2(\Omega)]^2} \right)\,d\sigma+E_{\varepsilon_n}(\overline{u}_n(t),\overline{\alpha}_n(t))
      \nonumber
      \\
      &\leq \frac{1}{2}\int_{m^s\tau_n}^{m_t\tau_n} \left( |\partial_t\alpha_n(\sigma)|^2_{L^2(\Omega)}+{\kappa}|\nabla\partial_t\alpha_n(\sigma)|^2_{[L^2(\Omega)]^2}\right)\,d\sigma
      \nonumber
      \\
      &\quad+\int_{m^s\tau_n}^{m_t\tau_n} \left(|\partial_tu_n(\sigma)|^2_{L^2(\Omega)}+{\mu}|\nabla\partial_tu_n(\sigma)|^2_{[L^2(\Omega)]^2} \right)\,d\sigma+E_{\varepsilon_n}(\overline{u}_n(t),\overline{\alpha}_n(t))
      \nonumber
      \\
      &\leq E_{\varepsilon_n}(\underline{u}_n(s),\underline{\alpha}_n(s)), \mbox{ for }s,t\in[0,T];s\leq t,\mbox{ and }n\in\N.\label{energy3}
    \end{align}
Now, by virtue of \eqref{conv_1} \eqref{conv_2}, \eqref{convJ} and \eqref{energy3}, letting $n\rightarrow\infty$ yields that:
\begin{align}
    &\frac{1}{2}\int_{s}^{t}|\partial_t\alpha(\sigma)|^2_{L^2(\Omega)}\,d\sigma+\frac{\kappa}{2}\int_{s}^{t}|\nabla\partial_t\alpha(\sigma)|^2_{[L^2(\Omega)]^2}\,d\sigma+\int_{s}^{t}|\partial_tu(\sigma)|^2_{L^2(\Omega)}\,d\sigma
    \nonumber
    \\
    &\quad+{\mu}\int_{s}^{t}|\nabla\partial_tu(\sigma)|^2_{[L^2(\Omega)]^2}\,d\sigma +E(u(t), \alpha(t)) \leq E(u(s), \alpha(s)),
    \label{conv_34}
    \\
    & \mbox{ for a.e. $ s \in [0, T) $ including $s = 0$, and a.e. $ t \in (s, T) $.}
\end{align}
{In addition, we can replace the phrase ``a.e. $ t \in (s, T) $'' in \eqref{conv_34}  by ``for any $ t \in [s, T] $''. In fact, we consider a sequence $\{t_n\}_{n\in\N}\subset (t,T)$ such that $t_n \rightarrow t$, and 
\begin{align}
  &\frac{1}{2}\int_{s}^{t_n}|\partial_t\alpha(\sigma)|^2_{L^2(\Omega)}\,d\sigma+\frac{\kappa}{2}\int_{s}^{t_n}|\nabla\partial_t\alpha(\sigma)|^2_{[L^2(\Omega)]^2}\,d\sigma+\int_{s}^{t_n}|\partial_tu(\sigma)|^2_{L^2(\Omega)}\,d\sigma
  \label{conv_35}
  \\
  &\quad+{\mu}\int_{s}^{t_n}|\nabla\partial_tu(\sigma)|^2_{[L^2(\Omega)]^2}\,d\sigma +E(u(t_n), \alpha(t_n)) \leq E(u(s), \alpha(s)),\mbox{ for all $ n\in\N $.}
\end{align}
Having in mind the lower semi-continuity of $E(u,\alpha)$ on $[L^2(\Omega)]^2$ and the convergence $[u(t_n),\alpha(t_n)] \rightarrow [u(t),\alpha(t)]$ in $[L^2(\Omega)]^2$, taking the lower limit of both sides of \eqref{conv_35} yields that 
\begin{align}
  &\frac{1}{2}\int_{s}^{t}|\partial_t\alpha(\sigma)|^2_{L^2(\Omega)}\,d\sigma+\frac{\kappa}{2}\int_{s}^{t}|\nabla\partial_t\alpha(\sigma)|^2_{[L^2(\Omega)]^2}\,d\sigma+\int_{s}^{t}|\partial_tu(\sigma)|^2_{L^2(\Omega)}\,d\sigma
    \nonumber
    \\
    &\quad+{\mu}\int_{s}^{t}|\nabla\partial_tu(\sigma)|^2_{[L^2(\Omega)]^2}\,d\sigma +E(u(t), \alpha(t)) 
    \\
  \leq& \liminf_{n\rightarrow\infty}\Biggl(\frac{1}{2}\int_{s}^{t_n}|\partial_t\alpha(\sigma)|^2_{L^2(\Omega)}\,d\sigma+\frac{\kappa}{2}\int_{s}^{t_n}|\nabla\partial_t\alpha(\sigma)|^2_{[L^2(\Omega)]^2}\,d\sigma+\int_{s}^{t_n}|\partial_tu(\sigma)|^2_{L^2(\Omega)}\,d\sigma
  \nonumber
  \\
  &\quad+{\mu}\int_{s}^{t_n}|\nabla\partial_tu(\sigma)|^2_{[L^2(\Omega)]^2}\,d\sigma +E(u(t_n), \alpha(t_n)) \Biggr)
  \\
  \leq& E(u(s), \alpha(s)). \label{conv_101}
\end{align}
Thus, we can derive the energy-inequality \eqref{ene-inq1} as a straightforward consequence of \eqref{conv_101}.}

  Finally, we prove that {$0 \leq u\leq 1$ in $\overline{Q}$}. Let us put $\psi= u(t) + [-u(t)]^+$ in (S1). From (A1) and (A2), we can see that
  \begin{align}
    &\frac{1}{2}\frac{d}{dt}(|[-u(t)]^+|^2_{L^2(\Omega)}+\mu|\nabla [-u(t)]^+|^2_{[L^2(\Omega)]^2})
    +\lambda(u_\mathrm{org},[-u(t)]^+)_{L^2(\Omega)}
    \nonumber
    \\
    &\quad+\int_{\Omega}\gamma(R(\alpha(t))\nabla u(t))\,dx
    \leq\int_{\Omega}\gamma (R(\alpha(t))\nabla[-u(t)]^-)\,dx,\mbox{ a.e. } t \in (0,T).
    \label{ugeq01}
  \end{align}
Here, since
\begin{align*}
    \int_{\Omega}\gamma (R(\alpha(t))\nabla[-u(t)]^-)\,dx= \int_{\{u\geq0\}}\gamma (R(\alpha(t))\nabla u(t))\,dx \leq\int_{\Omega}\gamma (R(\alpha(t))\nabla u(t))\,dx,
\end{align*}
we obtain
    \begin{align}\label{-uve+}
      \frac{d}{dt}\bigl(|[-u(t)]^+|^2_{L^2(\Omega)}+\mu|\nabla [-u(t)]^+|^2_{[L^2(\Omega)]^2}\bigr)\leq0,\mbox{ a.e. } t \in (0,T).
    \end{align}
    By applying Gronwall's lemma to \eqref{-uve+}, we obtain that
    \begin{gather}
      |[-u(t)]^+|^2_{L^2(\Omega)}+\mu|\nabla [-u(t)]^+|^2_{[L^2(\Omega)]^2}\leq|[-u_0]^+|^2_{L^2(\Omega)}+\mu|\nabla [-u_0]^+|^2_{[L^2(\Omega)]^2}=0, 
      \\
        \mbox{ for all }t\in[0,T], \mbox{ which implies $ u \geq 0 $ a.e. in $ Q $.}\label{dai01}
    \end{gather}

    Secondly, putting $\psi= u\land 1$ in (S1) and noting that $u-(u\land1)=[u-1]^+$, we can see that
    \begin{align}
      &(\partial_tu(t),[u-1]^+(t))_{L^2(\Omega)}+\lambda(u(t)-u_\mathrm{org},[u-1]^+(t))_{L^2(\Omega)}
      \nonumber
      \\
      &\quad+\mu(\nabla\partial_tu(t),\nabla[u-1]^+(t))_{[L^2(\Omega)]^2}+\nu\int_{\Omega}|\nabla u(t)|^{p-2}\nabla u(t)\cdot\nabla[u-1]^+(t)\,dx
      \nonumber
      \\
      &\quad+\int_{\Omega}\gamma(R(\alpha(t))\nabla u(t))\,dx\leq\int_{\Omega}\gamma (R(\alpha(t))\nabla(u\land1)(t))\,dx.
      \label{u-101}
    \end{align}
    Here, since
    \begin{gather*}
      \int_{\Omega}\gamma (R(\alpha(t))\nabla(u\land1)(t))\,dx=\int_{\{u\leq1\}}\gamma (R(\alpha(t))\nabla u(t))\,dx\leq\int_{\Omega}\gamma (R(\alpha(t))\nabla u(t))\,dx,
        \\[1ex]
      \partial_tu=\partial_t(u-1),\mbox{ and }\nabla u=\nabla(u-1),
        \\[1ex]
      1-u_{\mathrm{org}}\geq0\mbox{ a.e. in }\Omega, \mbox{ i.e., }(1-u_{\mathrm{org}},[u-1]^+(t))_{L^2(\Omega)}\geq0,
        \\[1ex]
      \nu\int_{\Omega}|\nabla u(t)|^{p-2}\nabla u(t)\cdot\nabla[u-1]^+(t)\,dx\geq \nu\int_{\{u\geq1\}}|\nabla(u-1)|^p\,dx \geq0,
    \end{gather*}
    we obtain
    \begin{equation}\label{1-uve+}
      \frac{d}{dt}\bigl(|[u(t)-1]^+|_{L^2(\Omega)}^2+\mu|\nabla[u(t)-1]^+|_{[L^2(\Omega)]^2}^2\bigr)\leq0,\mbox{ a.e. } t \in (0,T).
    \end{equation}
    By applying Gronwall's lemma to \eqref{1-uve+}, we obtain that
    \begin{gather}
      |[u(t)-1]^+|_{L^2(\Omega)}^2+\mu|\nabla[u(t)-1]^+|_{[L^2(\Omega)]^2}^2\leq|[u_0-1]^+|_{L^2(\Omega)}^2+\mu|\nabla[u_0-1]^+|_{[L^2(\Omega)]^2}^2=0,
      \\
        \mbox{ for all }t\in[0,T], \mbox{ which implies $ u \leq 1 $ a.e. in $ Q $.}\label{dai02}
    \end{gather}
    {Combining \eqref{dai01}, \eqref{dai02}, and the continuity $u\in C(\overline{Q})$, we arrive at $0\leq u\leq1 $ in $\overline{Q}$.}

Thus, we complete the proof of Main Theorem \ref{mainThm1}. \qed
\begin{rem}
  We use the diagonal method involving $\varepsilon$ and $\tau$. However, if $\gamma \in C^{1,1}(\mathbb{R}^2) \cap C^2(\R^2)$, then we can choose a sequence $\{\tau_n\}_{n\in\N}$ that is independent of $\varepsilon$. In other words, the existence of solutions can be shown by a subsequence of $\tau$ alone.
\end{rem}

\pagebreak

\section{Proof of Main Theorem 2}

Let $ [ u_k , \alpha_k ] ( k = 1 , 2 ) $ denote the solution of the system (S) with the initial condition $ [ u_1 ( 0 ) , \alpha_1 ( 0 ) ] = [ u_2 ( 0 ) , \alpha_2 ( 0 ) ] = [ u_0 , \alpha_0 ] $. Let us take a difference between two variational identities for $ \alpha_k $ and substitute $ \varphi : = ( \alpha_1 - \alpha_2 ) ( t ) $. {From (A2) and the additional assumption $\gamma\in C^{1,1}(\R^2)\cap C^2(\R^2)$, we can see that}
\begin{align}
  & \frac{1}{2}\frac{d}{dt}\bigl(|(\alpha_1-\alpha_2)(t)|^2_{L^2(\Omega)}+\kappa |\nabla(\alpha_1-\alpha_2)(t)|^2_{[L^2(\Omega)]^2}\bigr)+|\nabla(\alpha_1-\alpha_2)(t)|^2_{[L^2(\Omega)]^2}
  \nonumber
  \\
  &\leq4|\nabla^2\gamma|_{L^\infty(\R^2;\R^{2\times2})}\int_\Omega|\nabla u_1(t)|^2|(\alpha_1-\alpha_2)(t)|^2\,dx
  \nonumber
  \\
  & \quad+|\nabla^2\gamma|_{L^\infty(\R^2;\R^{2\times2})}\int_\Omega|\nabla u_1(t)||\nabla(u_1-u_2)(t)||(\alpha_1-\alpha_2)(t)|\,dx
  \nonumber
  \\
  &\quad+4|\nabla\gamma|_{L^\infty(\R^2;\R^2)}\int_\Omega|\nabla u_1(t)||(\alpha_1-\alpha_2)(t)|^2\,dx
  \\
  &\quad+|\nabla\gamma|_{L^\infty(\R^2;\R^2)}\int_\Omega|\nabla(u_1-u_2)(t)||(\alpha_1-\alpha_2)(t)|\,dx
  \nonumber
  \\
  &=:I_1+I_2+I_3+I_4.
\label{mT2-0}
\end{align}
Here, let us denote by $C_{H^1}^{L^{\frac{2p}{p -2}}}$ and $C_{H^1}^{L^{\frac{2p}{p -1}}}$ the constants of two-dimensional Sobolev embeddings $H^1(\Omega) \subset L^{\frac{2p}{p-2}}(\Omega)$ and $H^1(\Omega) \subset L^{\frac{2p}{p-1}}(\Omega)$, respectively. By Young's inequality, the integral terms $I_j$, $ j = 1, 2, 3, 4 $, in \eqref{mT2-0}  can be estimated as follows:
\begin{align}
  I_1 &\leq4|\nabla\gamma|_{W^{1,\infty}(\R^2;\R^2)}|\nabla u_1(t)|^2_{[L^p(\Omega)]^2}|(\alpha_1-\alpha_2)(t)|^2_{L^{\frac{2p}{p-2}}(\Omega)}
  \nonumber
  \\
  &\leq4(C_{H^1}^{L^{\frac{2p}{p-2}}})^2|\nabla\gamma|_{W^{1,\infty}(\R^2;\R^2)}|\nabla u_1(t)|^2_{[L^p(\Omega)]^2}|(\alpha_1-\alpha_2)(t)|^2_{H^1(\Omega)},
  \label{mT2-1}
    \\[2ex]
  I_2&\leq|\nabla\gamma|_{W^{1,\infty}(\R^2;\R^2)}|\nabla u_1(t)|_{[L^p(\Omega)]^2} \cdot 
    \\
    & \qquad \qquad \cdot |\nabla(u_1-u_2)(t)|_{[L^2(\Omega)]^2}|(\alpha_1-\alpha_2)(t)|_{L^{\frac{2p}{p-2}}(\Omega)}
  \nonumber
  \\
  &\leq\frac{C_{H^1}^{L^{\frac{2p}{p-2}}}}{2}|\nabla\gamma|_{W^{1,\infty}(\R^2;\R^2)}|\nabla u_1(t)|_{[L^p(\Omega)]^2} \cdot 
    \\
    & \qquad \qquad \cdot \bigl(|\nabla(u_1-u_2)(t)|_{[L^2(\Omega)]^2}^2+|(\alpha_1-\alpha_2)(t)|_{H^1(\Omega)}^2\bigr),
    \label{mT2-2}
    \\[2ex]
  I_3&\leq4|\nabla\gamma|_{W^{1,\infty}(\R^2;\R^2)}|\nabla u_1(t)|_{[L^p(\Omega)]^2}|(\alpha_1-\alpha_2)(t)|^2_{L^{\frac{2p}{p-1}}(\Omega)}
  \nonumber
  \\
  &\leq4(C_{H^1}^{L^{\frac{2p}{p-1}}})^2|\nabla\gamma|_{W^{1,\infty}(\R^2;\R^2)}|\nabla u_1(t)|_{[L^p(\Omega)]^2}|(\alpha_1-\alpha_2)(t)|^2_{H^1(\Omega)},
  \label{mT2-3}
\end{align}
and
\begin{align}
    I_4 & \leq \frac{1}{2}|\nabla\gamma|_{W^{1,\infty}(\R^2;\R^2)}\bigl(|\nabla(u_1-u_2)(t)|_{[L^2(\Omega)]^2}^2+|(\alpha_1-\alpha_2)(t)|_{H^1(\Omega)}^2\bigr).
  \label{mT2-10}
\end{align}
\noeqref{mT2-0,mT2-1,mT2-2,mT2-3,mT2-10}
Combining \eqref{mT2-0}--\eqref{mT2-10}, we arrive at
\begin{align}
  & \frac{d}{dt}\bigl(|(\alpha_1-\alpha_2)(t)|^2_{L^2(\Omega)}+\kappa|\nabla(\alpha_1-\alpha_2)(t)|^2_{[L^2(\Omega)]^2}\bigr)
  \nonumber
  \\
  & \quad\leq C_1\bigl(1+|\nabla u_1(t)|_{[L^p(\Omega)]^2}\bigr)^2\bigl(|(u_1-u_2)(t)|^2_{L^2(\Omega)}+\mu|\nabla(u_1-u_2)(t)|^2_{[L^2(\Omega)]^2}
  \label{mT2_a}
  \\
  &\qquad\quad+|(\alpha_1-\alpha_2)(t)|^2_{L^2(\Omega)}+\kappa|\nabla(\alpha_1-\alpha_2)(t)|^2_{[L^2(\Omega)]^2}\bigr), \mbox{ a.e. } t \in (0,T),
  \nonumber
\end{align}
where 
\begin{gather*}
    C_1:=\frac{8\bigl((C_{H^1}^{L^{\frac{2p}{p-2}}}+1)^2+(C_{H^1}^{L^{\frac{2p}{p-1}}})^2\bigr)|\nabla\gamma|_{W^{1,\infty}(\R^2;\R^2)}}{1\land\kappa\land\mu}.
\end{gather*}

On the other hand, by putting $ \psi = u_2 $ in the variational inequality for $ u_1 $, and $ \psi = u_1 $ in the one for $ u_2 $, and then summing the two inequalities, we obtain:
\begin{align*}
  &\frac{1}{2}\frac{d}{dt}\bigl(|(u_1-u_2)(t)|^2_{L^2(\Omega)}+\mu|\nabla(u_1-u_2)(t)|^2_{[L^2(\Omega)]^2}\bigr)+\lambda|(u_1-u_2)(t)|^2_{L^2(\Omega)}
  \\
  &+\int_\Omega|\nabla u_1(t)|^{p-2}\nabla u_1(t)\cdot\nabla(u_1-u_2)(t)\,dx +\int_\Omega|\nabla u_2(t)|^{p-2}\nabla u_2(t)\cdot\nabla(u_2-u_1)(t)\,dx 
  \\
  &-\int_\Omega\left[{^\top}R(\alpha_1(t))\nabla\gamma(R(\alpha_1(t))\nabla u_1(t))-{^\top}R(\alpha_2(t))\nabla\gamma(R(\alpha_2(t))\nabla u_2(t))\right]\cdot
  \\
  & \qquad\cdot\nabla(u_1-u_2)(t)\,dx\leq0,\mbox{ a.e. }t\in(0,T).
\end{align*}
Also, by using (A2), Tartar's inequality for $p$-Laplace operator (cf. \cite[Lemma 2.1]{MR2802983}), and H\"older's inequality, we obtain that
\begin{align}
  & \frac{1}{2}\frac{d}{dt}\bigl(|(u_1-u_2)(t)|^2_{L^2(\Omega)}+\mu|\nabla(u_1-u_2)(t)|^2_{[L^2(\Omega)]^2}\bigr)
  \nonumber
  \\
  &\leq4|\nabla\gamma|_{L^\infty(\R^2;\R^{2})}\int_\Omega|\nabla(u_1-u_2)(t)||(\alpha_1-\alpha_2)(t)|\,dx 
  \nonumber
  \\
  &\quad+|\nabla^2\gamma|_{L^\infty(\R^2;\R^{2\times2})}\int_\Omega|\nabla u_1(t)||\nabla(u_1-u_2)(t)||(\alpha_1-\alpha_2)(t)|\,dx
  \nonumber
  \\
  &\quad+|\nabla^2\gamma|_{L^\infty(\R^2;\R^{2\times2})}|\nabla(u_1-u_2)(t)|^2_{[L^2(\Omega)]^2}
  \nonumber
  \\
  &=: I_5+I_6+|\nabla^2\gamma|_{L^\infty(\R^2;\R^{2\times2})}|\nabla(u_1-u_2)(t)|^2_{[L^2(\Omega)]^2}.
  \label{mT2-4}
\end{align}
Here, by using the two-dimensional Sobolev embedding $H^1(\Omega) \subset L^{\frac{2p}{p-2}}(\Omega)$, and Young's inequality, the integral terms $I_5$ and $I_6$ in \eqref{mT2-4} can be estimated as follows:
\begin{align}
  I_5&\leq 4|\nabla\gamma|_{W^{1,\infty}(\R^2;\R^2)}\bigl(|\nabla(u_1-u_2)(t)|_{[L^2(\Omega)]^2}^2+|(\alpha_1-\alpha_2)(t)|_{H^1(\Omega)}^2\bigr),
  \label{mT2-20}
    \\[2ex]
  I_6&\leq|\nabla\gamma|_{W^{1,\infty}(\R^2;\R^2)}|\nabla u_1(t)|_{[L^p(\Omega)]^2} \cdot 
    \\
    & \qquad \qquad \cdot |\nabla(u_1-u_2)(t)|_{[L^2(\Omega)]^2}|(\alpha_1-\alpha_2)(t)|_{L^{\frac{2p}{p-2}}(\Omega)}
  \nonumber
  \\
  &\leq \frac{C_{H^1}^{L^{\frac{2p}{p-2}}}}{2}|\nabla\gamma|_{W^{1,\infty}(\R^2;\R^2)}|\nabla u_1(t)|_{[L^p(\Omega)]^2} \cdot
    \\
    & 
    \qquad \qquad \cdot \bigl(|\nabla(u_1-u_2)(t)|_{[L^2(\Omega)]^2}^2+|(\alpha_1-\alpha_2)(t)|_{H^1(\Omega)}^2\bigr).
  \label{mT2-5}
\end{align}
\noeqref{mT2-4,mT2-20,mT2-5}
Combining \eqref{mT2-4}--\eqref{mT2-5}, we arrive at
\begin{align}
  &\frac{d}{dt}\bigl(|(u_1-u_2)(t)|^2_{L^2(\Omega)}+\mu|\nabla(u_1-u_2)(t)|^2_{[L^2(\Omega)]^2}\bigr)
  \nonumber
  \\
  &\quad\leq C_2\bigl(1+|\nabla u_1(t)|_{[L^p(\Omega)]^2}\bigr)\bigl(|(u_1-u_2)(t)|^2_{L^2(\Omega)}+\mu|\nabla(u_1-u_2)(t)|^2_{[L^2(\Omega)]^2}
  \nonumber
  \\
   &\qquad\qquad\quad+|(\alpha_1-\alpha_2)(t)|^2_{L^2(\Omega)}+\kappa|\nabla(\alpha_1-\alpha_2)(t)|^2_{[L^2(\Omega)]^2}\bigr),\mbox{ a.e. }t\in(0,T),\label{mT2_u}
\end{align}
where
\begin{gather*}
    C_2:=\frac{6\bigl(C_{H^1}^{L^{\frac{2p}{p-1}}}+1\bigr)|\nabla\gamma|_{W^{1,\infty}(\R^2;\R^{2})}}{1\land\kappa\land\mu}.
\end{gather*}
 By adding inequalities \eqref{mT2_a} and \eqref{mT2_u}, we obtain that
 \begin{align}
  \frac{ d }{ dt } J ( t ) \leq C_* \bigl(1 + |\nabla u_1(t)|_{[L^p(\Omega)]^2} \bigr)^2 J ( t ), \mbox{ a.e. } t \in(0,T),
  \label{mT2_f}
 \end{align}
 {where} $ C_* : = C_1 + C_2 $, and
 \begin{gather*}
  J(t):=|(u_1-u_2)(t)|^2_{L^2(\Omega)}+\mu|\nabla(u_1-u_2)(t)|^2_{[L^2(\Omega)]^2}+|(\alpha_1-\alpha_2)(t)|^2_{L^2(\Omega)}
  \\
  +\kappa|\nabla(\alpha_1-\alpha_2)(t)|^2_{[L^2(\Omega)]^2},\mbox{ for all } t \in [0,T].
 \end{gather*}
 By applying Gronwall's lemma to \eqref{mT2_f}, we conclude that
 \begin{gather*}
  J(t)\leq\exp(C_*T(1+|u_1|_{L^\infty(0,T;W^{1,p}_0(\Omega))})^2)J(0),
     ~ \mbox{ for all } t \in [0,T].
     \label{conclude01}
 \end{gather*}

 {
    Moreover, since the above \eqref{conclude01} implies the uniqueness of solution to (S), we will see that the energy-inequality \eqref{ene-inq1} will hold for all $ 0 \leq s \leq t \leq T $. Indeed, once the unique solution $ [u, \alpha] \in [L^2(0, T; L^2(\Omega))]^2 $ to (S) is obtained, and $ s \in [0, T) $ is taken arbitrary, the uniqueness of solution allows us to verify \eqref{ene-inq1} ``for a.e. $ \widetilde{s} \in (s, T) $ including $ \widetilde{s} = s $ and any $ t \in [\widetilde{s}, T] $.'' This is established by applying the time-discretization approach in Section 4, to the case when the initial data of (S) is given by $ [u(s), \alpha(s)] \in W_0^{1, p}(\Omega) \times H_0^1(\Omega) $. 
}
\medskip

 Thus, we conclude the Main Theorem \ref{mainThm2}. \qed

\paragraph{Acknowledgments.}{
This work is supported by  Grant-in-Aid for Scientific Research (C) No. 20K03672, JSPS, and JST SPRING No. JPMJSP2109. {The first and third authors are also partially supported by NSF grant DMS-2408877, Air Force Office of Scientific Research (AFOSR) under Award NO: FA9550-22-1-0248, and Office of Naval Research (ONR) under Award NO: N00014-24-1-2147.}
}
\pagebreak

\providecommand{\href}[2]{#2}
\providecommand{\arxiv}[1]{\href{http://arxiv.org/abs/#1}{arXiv:#1}}
\providecommand{\url}[1]{\texttt{#1}}
\providecommand{\urlprefix}{URL }

\end{document}